\documentclass[numbook,envcountsame,smallcondensed,draft]{amsart}

\usepackage{cite,amsmath,amssymb,amsthm,latexsym,a4,graphicx,gastex}

\newtheorem{theorem}{Theorem}[section]

\newtheorem{proposition}[theorem]{Proposition}

\newtheorem{lemma}[theorem]{Lemma}

\newtheorem{corollary}[theorem]{Corollary}

\numberwithin{equation}{section}

\DeclareMathOperator{\Aut}{Aut}

\DeclareMathOperator{\con}{con}

\DeclareMathOperator{\End}{End}

\DeclareMathOperator{\Gr}{Gr}

\DeclareMathOperator{\Perm}{Perm}

\DeclareMathOperator{\Stab}{Stab}

\DeclareMathOperator{\Sub}{Sub}

\DeclareMathOperator{\var}{var}

\makeatletter

\renewcommand*\subjclass[2][2010]{\def\@subjclass{#2}\@ifundefined{subjclassname@#1}{\ClassWarning{\@classname}{Unknown edition (#1) of Mathematics Subject Classification; using '2010'.}}{\@xp\let\@xp\subjclassname\csname subjclassname@#1\endcsname}}

\makeatother

\begin{document}

\title[Cancellable elements of the lattices of varieties]{ Cancellable elements of the lattices of varieties\\
of semigroups and epigroups}
\thanks{All authors are supported by the Ministry of Education and Science of the Russian Federation (project 1.6018.2017/8.9) and by Russian Foundation for Basic Research (the first and the second authors by grant 18-31-00443; the third author by grant 17-01-00551).}

\author[V. Yu. Shaprynski\v{\i}]{Vyacheslav Yu. Shaprynski\v{\i}}

\author[D. V. Skokov]{Dmitry V. Skokov}

\author[B. M. Vernikov]{Boris M. Vernikov}
\address{Institute of Natural Sciences and Mathematics, Ural Federal University, Lenina str. 51, 620000 Ekaterinburg, Russia}
\email{vshapr@yandex.ru, dmitry.skokov@gmail.com, bvernikov@gmail.com}



\begin{abstract}
We completely determine all semigroup [epigroup] varieties that are cancellable elements of the lattice of all semigroup [respectively epigroup] varieties.
\end{abstract}

\maketitle

\section{Introduction and summary}
\label{intr}

The collection of all semigroup varieties forms a lattice under the following naturally defined operations: for varieties $\mathbf X$ and $\mathbf Y$, their \emph{join} $\mathbf{X\vee Y}$ is the variety generated by the set-theoretical union of $\mathbf X$ and $\mathbf Y$ (as classes of semigroups), while their \emph{meet} $\mathbf{X\wedge Y}$ coincides with the set-theoretical intersection of $\mathbf X$ and $\mathbf Y$. We denote the lattice of all semigroup varieties by $\mathbb{SEM}$. This lattice has been intensively studied during more than 50 years. Probably, the first result in this area is the description of atoms of the lattice $\mathbb{SEM}$ obtained by Kalicki and Scott in 1955~\cite{Kalicki-Scott-55}. Results obtained on the first stages of these investigations are observed in the surveys~\cite{Aizenstat-Boguta-79} and~\cite{Evans-71}. The later survey~\cite{Shevrin-Vernikov-Volkov-09} observes the situation in the area we discuss, which is close to the contemporary one. Note that the structure of the lattice $\mathbb{SEM}$ is very complex. This is confirmed, in particular, by the fact that this lattice contains an anti-isomorphic copy of the partition lattice over a countably infinite set~\cite{Burris-Nelson-71,Jezek-76}, whence $\mathbb{SEM}$ does not satisfy any non-trivial lattice identity.

In addition to the lattice $\mathbb{SEM}$, in this article we examine one more varietal lattice related to $\mathbb{SEM}$, namely, the lattice of all epigroup varieties. We recall the corresponding definitions. An element $x$ of a semigroup $S$ is called a \emph{group element} if $x$ lies in some subgroup of $S$. A semigroup $S$ is called an \emph{epigroup} if for any $x\in S$ there is a natural $n$ such that $x^n$ is a group element. Extensive information about epigroups can be found in the fundamental work by Shevrin~\cite{Shevrin-94} or his survey~\cite{Shevrin-05}. The class of epigroups is very wide. In particular, it includes all periodic semigroups (because some power of each element in a periodic semigroup $S$ lies in some finite cyclic subgroup of $S$) and all \emph{completely regular} semigroups (unions of groups).

Epigroups (along with completely regular, inverse, involutary semigroups etc.) can be considered as so-called \emph{unary semigroups}, i.e., semigroups equipped by an additional unary operation. A unary operation on an epigroup can be defined in the following way. Let $S$ be an epigroup and $x\in S$. Then some power of $x$ lies in a maximal subgroup of $S$. We denote this subgroup by $G_x$. The identity element of $G_x$ is denoted by $x^\omega$. It is well known (see~\cite{Shevrin-94} or~\cite{Shevrin-05}, for instance) that the element $x^\omega$ is well defined and $xx^\omega=x^\omega x\in G_x$. We denote the element inverse to $xx^\omega$ in $G_x$ by $\overline x$\,. The map $x\mapsto\,\overline x$ is the just mentioned unary operation on an epigroup $S$. The element $\overline x$ is called \emph{pseudoinverse} of $x$. Throughout this paper, we consider epigroups as algebras of type $(2,1)$ with the operations of multiplication and pseudoinversion. This naturally leads to the concept of varieties of epigroups as algebras with these two operations. 

The collection of all epigroup varieties forms a lattice under the operations defined by the same way as in the lattice $\mathbb{SEM}$ (see the first phrase of the article). We denote this lattice by $\mathbb{EPI}$. Note that the class of all epigroups is not an epigroup variety (because it is not closed under taking of infinite direct products), so the lattice $\mathbb{EPI}$ does not have the largest element. The class of all varieties of completely regular semigroups considered as unary semigroups forms an important sublattice of $\mathbb{EPI}$. This sublattice was intensively studied from the 1970s to the 1990s (see~\cite{Petrich-Reilly-99} or~\cite[Section~6]{Shevrin-Vernikov-Volkov-09}). An examination of the lattice $\mathbb{EPI}$ was initiated in~\cite{Shevrin-94}. An overview of the first results obtained in this area can be found in~\cite[Section~2]{Shevrin-Vernikov-Volkov-09}.

It is well known and can be easily checked that in every periodic epigroup the operation of pseudoinversion can be expressed in terms of multiplication (see~\cite{Shevrin-94} or~\cite{Shevrin-05}, for instance). This means that periodic varieties of epigroups can be identified with periodic varieties of semigroups. Thus, the lattices $\mathbb{SEM}$ and $\mathbb{EPI}$ have a big common sublattice, namely, the lattice $\mathbb{PER}$ of all periodic semigroup varieties. Results of the mentioned above article~\cite{Jezek-76} immediately imply that even the lattice $\mathbb{PER}$ contains an anti-isomorphic copy of the partition lattice over a countably infinite set. This means, in particular, that the lattice $\mathbb{EPI}$ also contains the dual to this partition lattice as a sublattice. Therefore, $\mathbb{EPI}$, as well as $\mathbb{SEM}$, does not satisfy any non-trivial lattice identity.

The absence of non-trivial identities in the lattices $\mathbb{SEM}$ and $\mathbb{EPI}$ makes it natural to examine the elements of these lattices with properties that are somehow connected with lattice identities. We take in mind so-called \emph{special elements} of different types in lattices. In the theory of lattices, special elements of many types are investigated. We recall definitions of those types of such elements that appear below. An element $x$ of a lattice $\langle L;\vee,\wedge\rangle$ is called \emph{neutral} if
$$
(\forall y,z\in L)\quad(x\vee y)\wedge(y\vee z)\wedge(z\vee x)=(x\wedge y)\vee(y\wedge z)\vee(z\wedge x).
$$
It is well known that an element $x$ is neutral if and only if, for all $y,z\in L$, the sublattice of $L$ generated by $x$, $y$ and $z$ is distributive (see~\cite[Theorem~254]{Gratzer-11}). Further, an element $x\in L$ is called
\begin{align*}
&\text{\emph{distributive} if}&&(\forall y,z\in L)\quad x\vee(y\wedge z)=(x\vee y)\wedge(x\vee z),\\
&\text{\emph{standard} if}&&(\forall y,z\in L)\quad(x\vee y)\wedge z=(x\wedge z)\vee(y\wedge z),\\
&\text{\emph{modular} if}&&(\forall y,z\in L)\quad y\le z\longrightarrow(x\vee y)\wedge z=(x\wedge z)\vee y,\\
&\text{\emph{cancellable} if}&&(\forall y,z\in L)\quad x\vee y=x\vee z\ \&\ x\wedge y=x\wedge z\longrightarrow y=z.
\end{align*}
It is easy to see that any standard element is cancellable, while any cancellable element is modular. Special elements play an important role in the general lattice theory (see~\cite[Section~III.2]{Gratzer-11}, for instance).

There are many articles devoted to special elements of different types in the lattice $\mathbb{SEM}$. An overview of results in this subject published before 2015 can be found in the survey~\cite{Vernikov-15}. A number of results about special elements of different types in the lattice $\mathbb{EPI}$ were obtained in~\cite{Shaprynskii-Skokov-Vernikov-16,Skokov-15,Skokov-16}. We mention here only two results (one for semigroup and one for epigroup cases) closely related with this article. In~\cite[Theorem~3.1]{Vernikov-07}, commutative semigroup varieties that are modular elements of the lattice $\mathbb{SEM}$ are completely determined. An analogous result concerning epigroup varieties is proved in~\cite[Theorem~1.3]{Shaprynskii-Skokov-Vernikov-16}. 

Until recently nothing was known about cancellable elements in the lattices $\mathbb{SEM}$ and $\mathbb{EPI}$. However, now the situation has changed. It is proved in~\cite{Gusev-Skokov-Vernikov-18} that, for commutative semigroup varieties, the properties of being cancellable and modular elements in $\mathbb{SEM}$ are equivalent. An analogous result for the lattice $\mathbb{EPI}$ is verified in~\cite{Skokov-18}. These claims together with the results mentioned at the end of the previous paragraph give a complete description of commutative semigroup [epigroup] varieties that are cancellable elements of the lattice $\mathbb{SEM}$ [respectively $\mathbb{EPI}$]. For arbitrary semigroup [epigroup] varieties the properties of being cancellable and modular elements in the lattice $\mathbb{SEM}$ [respectively $\mathbb{EPI}$] are not equivalent. This is verified by the second and the third author for $\mathbb{SEM}$ in~\cite{Skokov-Vernikov-19} and by the second author for $\mathbb{EPI}$ (unpublished). In the present article we finish investigations started in~\cite{Gusev-Skokov-Vernikov-18,Skokov-18,Skokov-Vernikov-19} and completely determine all cancellable elements in $\mathbb{SEM}$ and $\mathbb{EPI}$. 

To formulate the main results of the article, we need some definitions and notation. Elements of the free unary semigroup are called \emph{words}. Words unlike \emph{letters} (elements of alphabet) are written in bold. A word that does not contain a unary operation is called a \emph{semigroup word}. It is natural to consider semigroup words as elements of the free semigroup. We connect two parts of an identity by the symbol~$\approx$, while the symbol~$=$ denotes, among other things, the equality relation on the free [unary] semigroup. Note that a semigroup $S$ satisfies the identity system $\mathbf wx\approx x\mathbf{w\approx w}$ where the letter $x$ does not occur in the word $\mathbf w$ if and only if $S$ contains a zero element~0 and all values of $\mathbf w$ in $S$ are equal to~0. We adopt the usual convention of writing $\mathbf w\approx 0$ as a short form of such a system and referring to the expression $\mathbf w\approx 0$ as to a single identity. Identities of the form $\mathbf w\approx 0$ and varieties given by such identities are called 0-\emph{reduced}. By $S_m$ we denote the full symmetric group on the set $\{1,2,\dots,m\}$. The identity
\begin{equation}
\label{permut id}
x_1x_2\cdots x_m\approx x_{1\pi}x_{2\pi}\cdots x_{m\pi}
\end{equation}
where $\pi\in S_m$ is denoted by $p_m[\pi]$. If the permutation $\pi$ is non-trivial then this identity is called \emph{permutational}. The number $m$ is called a \emph{length} of this identity. Let $\mathbf T$, $\mathbf{SL}$ and $\mathbf{SEM}$ be the trivial semigroup variety, the variety of all semilattices and the variety of all semigroups respectively. The semigroup variety given by the identity system $\Sigma$ is denoted by $\var\Sigma$. Put
\begin{align*}
\mathbf X_{\infty,\infty}={}&\var\{x^2y\approx xyx\approx yx^2\approx 0\},\\
\mathbf X_{m,\infty}={}&\mathbf X_{\infty,\infty}\wedge\var\{p_m[\pi]\mid\pi\in S_m\}\text{ where }2\le m<\infty,\\
\mathbf X_{m,n}={}&\mathbf X_{m,\infty}\wedge\var\{x_1x_2\cdots x_n\approx 0\}\text{ where }2\le m\le n<\infty,\\
\mathbf Y_{m,n}={}&\mathbf X_{m,n}\wedge\var\{x^2\approx 0\}\text{ where }2\le m\le n\le\infty.
\end{align*}
Note that the varieties $\mathbf T$, $\mathbf{SL}$, $\mathbf X_{m,n}$ and $\mathbf Y_{m,n}$ with $2\le m\le n\le\infty$ are periodic. Whence, they can be considered both as semigroup varieties and as epigroup ones. 

The first main result of the article is the following

\begin{theorem}
\label{main semigroup}
A semigroup variety $\mathbf V$ is a cancellable element of the lattice $\mathbb{SEM}$ if and only if either $\mathbf{V=SEM}$ or $\mathbf{V=M\vee N}$ where $\mathbf M$ is one of the varieties $\mathbf T$ or $\mathbf{SL}$, while $\mathbf N$ is one of the varieties $\mathbf T$, $\mathbf X_{m,n}$ or $\mathbf Y_{m,n}$ with $2\le m\le n\le\infty$.
\end{theorem}

It is easy to see that any 0-reduced semigroup variety is a modular element of $\mathbb{SEM}$ (see~\cite[Proposition~1.1]{Jezek-McKenzie-93}\footnote{We note that the paper~\cite{Jezek-McKenzie-93} has dealt with the lattice of equational theories of semigroups, that is, the dual of $\mathbb{SEM}$ rather than the lattice $\mathbb{SEM}$ itself. When reproducing results from~\cite{Jezek-McKenzie-93}, we adapt them to the terminology of the present article.} or~\cite[Theorem~3.8]{Vernikov-15}). It is asked in~\cite[Question~3.3]{Gusev-Skokov-Vernikov-18}, whether an arbitrary 0-reduced variety is a cancellable element of $\mathbb{SEM}$. Theorem~\ref{main semigroup} shows that the answer is negative. 

To make another comment to Theorem~\ref{main semigroup}, we need a few new definitions and notation. If $\mathbf u$ is a semigroup word then its length is denoted by $\ell(\mathbf u)$. If, otherwise, $\mathbf u$ is a non-semigroup word then we put $\ell(\mathbf u)=\infty$. For an arbitrary word $\mathbf u$, we denote by $\con(\mathbf u)$ the \emph{content} of $\mathbf u$, i.e., the set of all letters occurring in $\mathbf u$. An identity $\mathbf{u\approx v}$ is called \emph{substitutive} if $\mathbf u$ and $\mathbf v$ are semigroup words, $\con(\mathbf u)=\con(\mathbf v)$ and the word $\mathbf v$ is obtained from $\mathbf u$ by renaming of letters. Clearly, any permutational identity is a substitutive one. Other examples of substitutive identities are, for instance, $x^2y^2\approx y^2x^2$, $x^2yz\approx y^2zx$ or $xyx\approx yxy$. It is verified in~\cite[Theorem~2.5]{Vernikov-07} that if a nil-variety of semigroups $\mathbf N$ is a modular element of the lattice $\mathbb{SEM}$ then $\mathbf N$ can be given by 0-reduced and substitutive identities only. Theorem~\ref{main semigroup} shows that for nil-varieties that are cancellable elements of $\mathbb{SEM}$ a stronger claim is true: such varieties can be given by 0-reduced and permutational identities only.

The second main result of the article is the following

\begin{theorem}
\label{main epigroup}
An epigroup variety $\mathbf V$ is a cancellable element of the lattice $\mathbb{EPI}$ if and only if $\mathbf{V=M\vee N}$ where $\mathbf M$ is one of the varieties $\mathbf T$ or $\mathbf{SL}$, while $\mathbf N$ is one of the varieties $\mathbf T$, $\mathbf X_{m,n}$ or $\mathbf Y_{m,n}$ with $2\le m\le n\le\infty$.
\end{theorem}

As in the articles~\cite{Shaprynskii-Skokov-Vernikov-16,Skokov-15,Skokov-16,Skokov-18}, the formulation of the result concerning epigroup varieties turns out to be quite similar to its semigroup analogue. However, as we will see below, the proof of Theorem~\ref{main epigroup} essentially uses an epigroup specific.

By analogy with the mentioned above semigroup fact, it is verified in~\cite[Theorem~1.2]{Shaprynskii-Skokov-Vernikov-16} that if a nil-variety of epigroups $\mathbf N$ is a modular element of the lattice $\mathbb{EPI}$ then $\mathbf N$ can be given by 0-reduced and substitutive identities only. Theorem~\ref{main epigroup} shows that, as well as in the semigroup case, for nil-varieties that are cancellable elements of $\mathbb{EPI}$ a stronger claim is true: such varieties can be given by 0-reduced and permutational identities only.

Theorems~\ref{main semigroup} and~\ref{main epigroup} immediately imply

\begin{corollary}
\label{semigroup=epigroup}
A periodic semigroup variety is a cancellable element of the lattice $\mathbb{SEM}$ if and only if it is a cancellable element of the lattice $\mathbb{EPI}$.\qed
\end{corollary}

We does not know, whether the set of all cancellable elements of an arbitrary lattice $L$ forms a sublattice of $L$. Theorems~\ref{main semigroup} and~\ref{main epigroup} imply immediately that for lattice $\mathbb{SEM}$ and $\mathbb{EPI}$ the answer is affirmative. Moreover, the following is true.

\begin{corollary}
\label{sublattice}
Let $L$ be one of the lattices $\mathbb{SEM}$ or $\mathbb{EPI}$. The class of all cancellable elements of $L$ forms a countably infinite distributive sublattice of $L$. 
\end{corollary}

\begin{proof}
The varieties $\mathbf T$, $\mathbf X_{m,n}$ and $\mathbf Y_{m,n}$ with $2\le m\le n\le\infty$ form a lattice shown in Fig.~\ref{SEM canc diagram}. Evidenly, this lattice is countably infinite and distributive. The whole lattice of cancellable elements of $\mathbb{EPI}$ is the direct product of the lattice shown in Fig.~\ref{SEM canc diagram} and the 2-element chain (consisting of the varieties $\mathbf T$ and $\mathbf{SL}$). Finally, the whole lattice of cancellable elements of $\mathbb{SEM}$ is the previous lattice with the new greatest element (the variety $\mathbf{SEM}$) adjoined.
\end{proof}

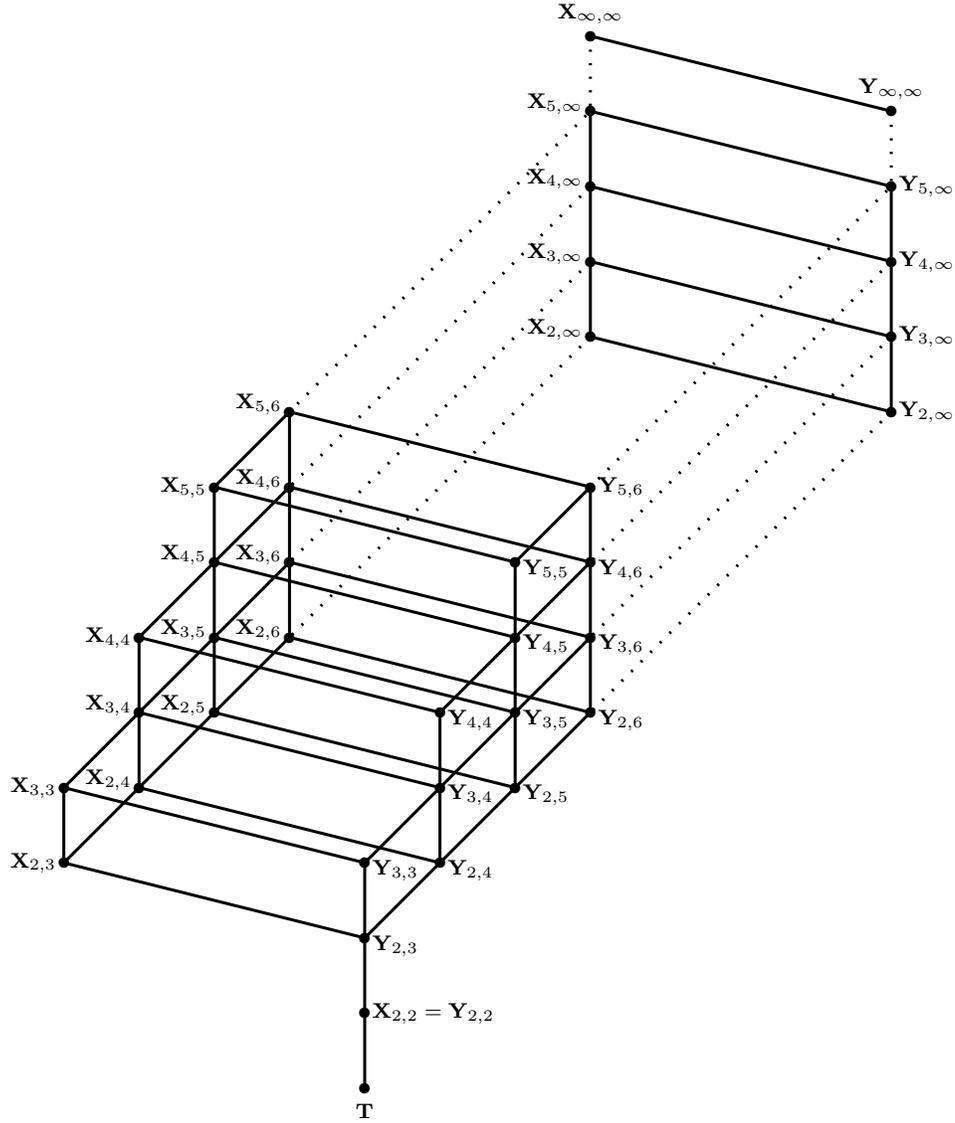
\begin{figure}[tbh]
\begin{center}
\unitlength=1mm
\linethickness{0.4pt}
\begin{picture}(127,150)
\put(9,35){\circle*{1.33}}
\put(9,45){\circle*{1.33}}
\put(19,45){\circle*{1.33}}
\put(19,55){\circle*{1.33}}
\put(19,65){\circle*{1.33}}
\put(29,55){\circle*{1.33}}
\put(29,65){\circle*{1.33}}
\put(29,75){\circle*{1.33}}
\put(29,85){\circle*{1.33}}
\put(39,65){\circle*{1.33}}
\put(39,75){\circle*{1.33}}
\put(39,85){\circle*{1.33}}
\put(39,95){\circle*{1.33}}
\put(49,5){\circle*{1.33}}
\put(49,15){\circle*{1.33}}
\put(49,25){\circle*{1.33}}
\put(49,35){\circle*{1.33}}
\put(59,35){\circle*{1.33}}
\put(59,45){\circle*{1.33}}
\put(59,55){\circle*{1.33}}
\put(69,45){\circle*{1.33}}
\put(69,55){\circle*{1.33}}
\put(69,65){\circle*{1.33}}
\put(69,75){\circle*{1.33}}
\put(79,55){\circle*{1.33}}
\put(79,65){\circle*{1.33}}
\put(79,75){\circle*{1.33}}
\put(79,85){\circle*{1.33}}
\put(79,105){\circle*{1.33}}
\put(79,115){\circle*{1.33}}
\put(79,125){\circle*{1.33}}
\put(79,135){\circle*{1.33}}
\put(79,145){\circle*{1.33}}
\put(119,95){\circle*{1.33}}
\put(119,105){\circle*{1.33}}
\put(119,115){\circle*{1.33}}
\put(119,125){\circle*{1.33}}
\put(119,135){\circle*{1.33}}
\gasset{AHnb=0,linewidth=.4}
\drawline[dash={0.3 1.5}{0}](39,65)(79,105)
\drawline[dash={0.3 1.5}{0}](39,75)(79,115)
\drawline[dash={0.3 1.5}{0}](39,85)(79,125)
\drawline[dash={0.3 1.5}{0}](39,95)(79,135)(79,145)
\drawline[dash={0.3 1.5}{0}](79,55)(119,95)
\drawline[dash={0.3 1.5}{0}](79,65)(119,105)
\drawline[dash={0.3 1.5}{0}](79,75)(119,115)
\drawline[dash={0.3 1.5}{0}](79,85)(119,125)(119,135)
\drawline(49,5)(49,35)(9,45)(9,35)(49,25)(79,55)(79,85)(39,95)(39,65)(9,35)
\drawline(9,45)(39,75)(79,65)(49,35)
\drawline(29,55)(69,45)(69,75)(29,85)
\drawline(29,55)(29,85)(39,95)
\drawline(19,55)(59,45)
\drawline(19,65)(59,55)
\drawline(29,65)(69,55)
\drawline(29,75)(69,65)
\drawline(39,65)(79,55)
\drawline(69,75)(79,85)
\drawline(79,115)(119,105)
\drawline(79,125)(119,115)
\drawline(79,145)(119,135)
\drawpolygon(19,45)(19,65)(39,85)(79,75)(59,55)(59,35)
\drawpolygon(79,105)(79,135)(119,125)(119,95)
\footnotesize
\put(49,2){\makebox(0,0)[cc]{$\mathbf T$}}
\put(50,15){\makebox(0,0)[lc]{$\mathbf X_{2,2}=\mathbf Y_{2,2}$}}
\put(78,106){\makebox(0,0)[rc]{$\mathbf X_{2,\infty}$}}
\put(8,35){\makebox(0,0)[rc]{$\mathbf X_{2,3}$}}
\put(8,45){\makebox(0,0)[rc]{$\mathbf X_{3,3}$}}
\put(78,116){\makebox(0,0)[rc]{$\mathbf X_{3,\infty}$}}
\put(18,46){\makebox(0,0)[rc]{$\mathbf X_{2,4}$}}
\put(18,56){\makebox(0,0)[rc]{$\mathbf X_{3,4}$}}
\put(18,65){\makebox(0,0)[rc]{$\mathbf X_{4,4}$}}
\put(78,126){\makebox(0,0)[rc]{$\mathbf X_{4,\infty}$}}
\put(28,56){\makebox(0,0)[rc]{$\mathbf X_{2,5}$}}
\put(28,66){\makebox(0,0)[rc]{$\mathbf X_{3,5}$}}
\put(28,76){\makebox(0,0)[rc]{$\mathbf X_{4,5}$}}
\put(28,85){\makebox(0,0)[rc]{$\mathbf X_{5,5}$}}
\put(78,136){\makebox(0,0)[rc]{$\mathbf X_{5,\infty}$}}
\put(38,66){\makebox(0,0)[rc]{$\mathbf X_{2,6}$}}
\put(38,76){\makebox(0,0)[rc]{$\mathbf X_{3,6}$}}
\put(38,86){\makebox(0,0)[rc]{$\mathbf X_{4,6}$}}
\put(38,96){\makebox(0,0)[rc]{$\mathbf X_{5,6}$}}
\put(79,148){\makebox(0,0)[cc]{$\mathbf X_{\infty,\infty}$}}
\put(120,95){\makebox(0,0)[lc]{$\mathbf Y_{2,\infty}$}}
\put(50,24){\makebox(0,0)[lc]{$\mathbf Y_{2,3}$}}
\put(50,34){\makebox(0,0)[lc]{$\mathbf Y_{3,3}$}}
\put(120,105){\makebox(0,0)[lc]{$\mathbf Y_{3,\infty}$}}
\put(60,34){\makebox(0,0)[lc]{$\mathbf Y_{2,4}$}}
\put(60,44){\makebox(0,0)[lc]{$\mathbf Y_{3,4}$}}
\put(60,54){\makebox(0,0)[lc]{$\mathbf Y_{4,4}$}}
\put(120,115){\makebox(0,0)[lc]{$\mathbf Y_{4,\infty}$}}
\put(70,44){\makebox(0,0)[lc]{$\mathbf Y_{2,5}$}}
\put(70,54){\makebox(0,0)[lc]{$\mathbf Y_{3,5}$}}
\put(70,64){\makebox(0,0)[lc]{$\mathbf Y_{4,5}$}}
\put(70,74){\makebox(0,0)[lc]{$\mathbf Y_{5,5}$}}
\put(120,125){\makebox(0,0)[lc]{$\mathbf Y_{5,\infty}$}}
\put(80,54){\makebox(0,0)[lc]{$\mathbf Y_{2,6}$}}
\put(80,64){\makebox(0,0)[lc]{$\mathbf Y_{3,6}$}}
\put(80,74){\makebox(0,0)[lc]{$\mathbf Y_{4,6}$}}
\put(80,85){\makebox(0,0)[lc]{$\mathbf Y_{5,6}$}}
\put(119,138){\makebox(0,0)[cc]{$\mathbf Y_{\infty,\infty}$}}
\end{picture}
\caption{The lattice of cancellable nil-varieties}
\label{SEM canc diagram}
\end{center}
\end{figure}

The article consists of four sections. Section~\ref{prel} contains some preliminary results. In Sections~\ref{only if part} and~\ref{if part} we verify respectively the ``only if'' part and the ``if'' part of both the theorems.

\section{Preliminaries}
\label{prel}

\subsection{The join with a neutral atom}
\label{join with na}

If $L$ is a lattice and $a\in L$ then we denote by $[a)$ the \emph{principal filter} of $L$ generated by $a$. In other words, $[a)=\{x\in L\mid x\ge a\}$.

\begin{lemma}[\!\!{\mdseries\cite[Lemmas~2.1 and~2.2]{Gusev-Skokov-Vernikov-18}}]
\label{neutral atom}
Suppose that $L$ is a lattice and $a\in L$ is an atom and a neutral element. For an element $x\in L$, the following are equivalent:
\begin{itemize}
\item[\textup{a)}] $x$ is cancellable;
\item[\textup{b)}] $x\vee a$ is cancellable;
\item[\textup{c)}] the implication
$$
x\vee y=x\vee z\ \&\ x\wedge y=x\wedge z\rightarrow y=z
$$
is true for any $y,z\in[a)$.\qed
\end{itemize}
\end{lemma}

It is generally known that the variety $\mathbf{SL}$ is an atom of the lattice $\mathbb{SEM}$ and therefore, of the lattice $\mathbb{EPI}$. Further, $\mathbf{SL}$ is a neutral element of $\mathbb{SEM}$~\cite[Proposition~4.1]{Volkov-05} and $\mathbb{EPI}$~\cite[Theorem~1.1]{Shaprynskii-Skokov-Vernikov-16}. These facts and Lemma~\ref{neutral atom} imply immediately

\begin{corollary}
\label{join with SL}
For a variety of semigroups \textup[epigroups\textup] $\mathbf V$, the following are equivalent:
\begin{itemize}
\item[\textup{a)}] $\mathbf V$ is a cancellable element of the lattice $\mathbb{SEM}$ \textup[respectively $\mathbb{EPI}$\textup];
\item[\textup{b)}] $\mathbf{V\vee SL}$ is a cancellable element of the lattice $\mathbb{SEM}$ \textup[respectively $\mathbb{EPI}$\textup];
\item[\textup{c)}] the implication
$$
\mathbf{V\vee U=V\vee W\ \&\ V\wedge U=V\wedge W\rightarrow U=W}
$$
is true for any varieties of semigroups \textup[epigroups\textup] $\mathbf{U,W\supseteq SL}$.\qed
\end{itemize}
\end{corollary}

\subsection{Cancellable elements in subgroup lattices of finite symmetric groups}
\label{cancellable subgroups}

If $n$ is a natural number and $1\le i\le n$ then we denote by $\Stab_n(i)$ the set of all permutations $\pi\in S_n$ with $i\pi=i$. Obviously, $\Stab_n(i)$ is a subgroup in $S_n$. Let $T$ be the trivial group, $T_{ij}$ be the group generated by the transposition $(ij)$, $C_{ijk}$ and $C_{ijk\ell}$ be the groups generated by the cycles $(ijk)$ and $(ijk\ell)$ respectively, $P_{ij,k\ell}$ be the group generated by the disjoint transpositions $(ij)$ and $(k\ell)$, $A_n$ be the alternating subgroup of $S_n$ and $V_4$ be the Klein four-group. The subgroup lattice of the group $G$ is denoted by $\Sub(G)$. We need to know the structure of the lattices $\Sub(S_3)$ and $\Sub(S_4)$. It is generally known and easy to check that the first of these two lattices has the form shown in Fig.~\ref{Sub(S_3)}. Direct routine calculations allow to verify that the lattice $\Sub(S_4)$ is as shown in Fig.~\ref{Sub(S_4)}. 

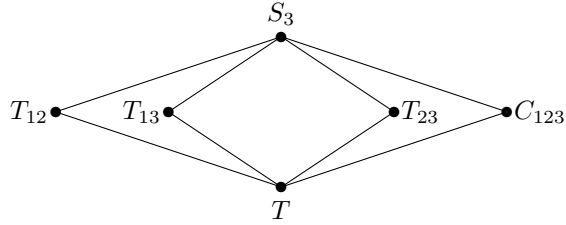
\begin{figure}[tbh]
\unitlength=1mm
\special{em:linewidth 0.4pt}
\linethickness{0.4pt}
\begin{center}
\begin{picture}(73,29)
\put(6,14){\circle*{1.33}}
\put(21,14){\circle*{1.33}}
\put(36,4){\circle*{1.33}}
\put(36,24){\circle*{1.33}}
\put(51,14){\circle*{1.33}}
\put(66,14){\circle*{1.33}}
\gasset{AHnb=0}
\drawline(36,4)(6,14)(36,24)(21,14)(36,4)(51,14)(36,24)(66,14)(36,4)
\put(5,14){\makebox(0,0)[rc]{$T_{12}$}}
\put(20,14){\makebox(0,0)[rc]{$T_{13}$}}
\put(52,14){\makebox(0,0)[lc]{$T_{23}$}}
\put(67,14){\makebox(0,0)[lc]{$C_{123}$}}
\put(36,27){\makebox(0,0)[cc]{$S_3$}}
\put(36,1){\makebox(0,0)[cc]{$T$}}
\end{picture}
\end{center}
\caption{the lattice $\Sub(S_3)$}
\label{Sub(S_3)}
\end{figure}

\begin{figure}[tbh]
\unitlength=0.85mm
\special{em:linewidth 0.4pt}
\linethickness{0.4pt}
\begin{center}
\begin{picture}(135,70)
\put(73,5){\circle*{1.56}}
\put(13,20){\circle*{1.56}}
\put(23,20){\circle*{1.56}}
\put(33,20){\circle*{1.56}}
\put(43,20){\circle*{1.56}}
\put(53,20){\circle*{1.56}}
\put(63,20){\circle*{1.56}}
\put(73,20){\circle*{1.56}}
\put(83,20){\circle*{1.56}}
\put(93,20){\circle*{1.56}}
\put(103,20){\circle*{1.56}}
\put(113,20){\circle*{1.56}}
\put(123,20){\circle*{1.56}}
\put(133,20){\circle*{1.56}}
\put(3,35){\circle*{1.56}}
\put(13,35){\circle*{1.56}}
\put(23,35){\circle*{1.56}}
\put(33,35){\circle*{1.56}}
\put(43,35){\circle*{1.56}}
\put(53,35){\circle*{1.56}}
\put(63,35){\circle*{1.56}}
\put(73,35){\circle*{1.56}}
\put(83,35){\circle*{1.56}}
\put(93,35){\circle*{1.56}}
\put(103,35){\circle*{1.56}}
\put(73,50){\circle*{1.56}}
\put(83,50){\circle*{1.56}}
\put(93,50){\circle*{1.56}}
\put(113,50){\circle*{1.56}}
\put(73,65){\circle*{1.56}}
\gasset{AHnb=0}
\drawline(73,5)(13,20)(3,35)(23,20)(73,5)(33,20)(3,35)(73,65)(73,5)(83,20)(103,35)(93,20)(73,5)(103,20)(113,50)(93,20)(73,5)(123,20)(113,50)(133,20)(73,5)(43,20)(23,35)(73,65)(13,35)(13,20)(73,50)(103,35)(83,50)(73,65)(93,50)(103,35)(113,50)(73,65)(33,35)(33,20)(93,50)(93,20)(63,35)(63,20)(23,35)(123,20)
\drawline(23,35)(23,20)(83,50)(83,20)(53,35)(53,20)(73,5)(63,20)(13,35)(53,20)(33,35)(43,20)(43,35)(73,20)(103,35)
\drawline(33,20)(33,35)(133,20)
\drawline(13,35)(113,20)(73,5)
\drawline(113,20)(113,50)
\drawline(3,35)(103,20)
\gasset{AHnb=1}
\drawline(5,42)(3,36)
\drawline(17,48)(13,36)
\drawline(32,54)(23,36)
\drawline(47,60)(33,36)
\footnotesize
\put(113,53){\makebox(0,0)[cc]{$A_4$}}
\put(128,17){\makebox(0,0)[lc]{$C_{123}$}}
\put(87,38){\makebox(0,0)[rc]{$C_{1234}$}}
\put(118,17){\makebox(0,0)[lc]{$C_{124}$}}
\put(97,38){\makebox(0,0)[rc]{$C_{1243}$}}
\put(77,38){\makebox(0,0)[rc]{$C_{1324}$}}
\put(108,17){\makebox(0,0)[lc]{$C_{134}$}}
\put(98,17){\makebox(0,0)[lc]{$C_{234}$}}
\put(76,17){\makebox(0,0)[rc]{$P_{12,34}$}}
\put(86,17){\makebox(0,0)[rc]{$P_{13,24}$}}
\put(96,17){\makebox(0,0)[rc]{$P_{14,23}$}}
\put(73,68){\makebox(0,0)[cc]{$S_4$}}
\put(7,43){\makebox(0,0)[cc]{$\Stab_4(1)$}}
\put(19,49){\makebox(0,0)[cc]{$\Stab_4(2)$}}
\put(34,55){\makebox(0,0)[cc]{$\Stab_4(3)$}}
\put(49,61){\makebox(0,0)[cc]{$\Stab_4(4)$}}
\put(73,2){\makebox(0,0)[cc]{$T$}}
\put(43,17){\makebox(0,0)[cc]{$T_{12}$}}
\put(53,17){\makebox(0,0)[cc]{$T_{13}$}}
\put(63,17){\makebox(0,0)[cc]{$T_{14}$}}
\put(33,17){\makebox(0,0)[cc]{$T_{23}$}}
\put(23,17){\makebox(0,0)[cc]{$T_{24}$}}
\put(13,17){\makebox(0,0)[cc]{$T_{34}$}}
\put(104,32){\makebox(0,0)[cc]{$V_4$}}
\end{picture}
\end{center}
\caption{the lattice $\Sub(S_4)$}
\label{Sub(S_4)}
\end{figure}
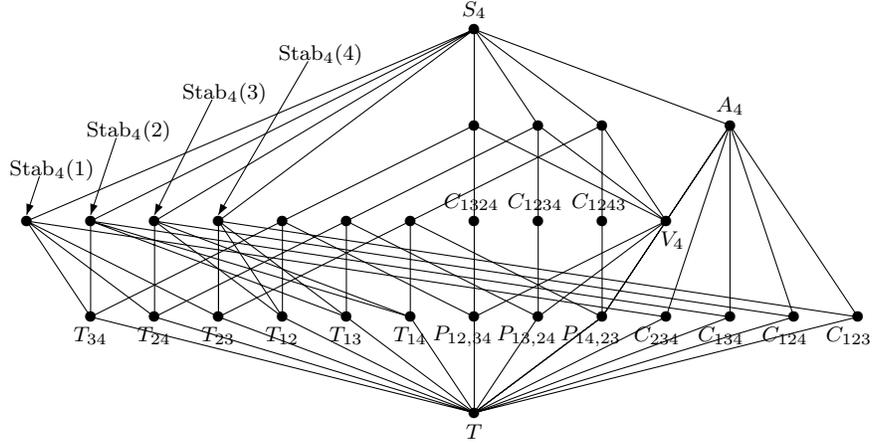

We need the following

\begin{lemma}
\label{subgroups cancellable}
A subgroup $G$ of the group $S_n$ is a cancellable element of the lattice $\Sub(S_n)$ if and only if either $G=T$ or $G=S_n$.
\end{lemma}

\begin{proof}
If $n\le 2$ then $S_n$ does not contain subgroups different from $T$ and $S_n$. If $n=3$ then the desirable conclusion immediately follows from Fig.~\ref{Sub(S_3)}. Let now $n\ge 4$ and $G$ be a non-singleton proper subgroup of $S_n$. Suppose that $G$ is a cancellable and therefore, modular element of $\Sub(S_n)$. Modular elements of the lattice $\Sub(S_n)$ for an arbitrary $n$ are completely determined in~\cite[Propositions~3.1,~3.7 and~3.8]{Jezek-81}. If $n=4$ then $G\supseteq V_4$ by~\cite[Proposition~3.8]{Jezek-81}, while if $n\ge 5$ then $G\supseteq A_n$ by~\cite[Proposition~3.1]{Jezek-81}. Clearly, it suffices to verify that there are at least two complements to $G$ in $\Sub(S_n)$. Suppose that $G\supseteq A_n$. Then $G=A_n$ because $A_n$ is a maximal proper subgroup in $S_n$. Then all subgroups of the form $T_{ij}$ are complements to $G$ in $\Sub(S_n)$. It remains to consider the case when $n=4$, $V_4\subseteq G\subset S_4$ and $G\ne A_4$. Fig.~\ref{Sub(S_4)} implies that either $G=V_4$ or $G=V_4\vee P_{ij,k\ell}$ for some disjoint transpositions $(ij)$ and $(k\ell)$. If $G=V_4$ then all subgroups of the form $\Stab_4(i)$ are complements to $G$ in $\Sub(S_4)$. Finally, if $G=V_4\vee P_{ij,k\ell}$ then subgroups $T_{ik}$ and $T_{j\ell}$ are complements to $G$ in $\Sub(S_4)$. 
\end{proof}

\subsection{Modular elements of the lattices $\mathbb{SEM}$ and $\mathbb{EPI}$}
\label{mod}

The following claim gives a strong necessary condition for a semigroup [an epigroup] variety to be a modular element in the lattice $\mathbb{SEM}$ [respectively $\mathbb{EPI}$]. Recall that a semigroup variety is called \emph{proper} if it differs from the variety of all semigroups.

\begin{proposition}
\label{mod nec}
If a proper variety of semigroups \textup[a variety of epigroups\textup] $\mathbf V$ is a modular element of the lattice $\mathbb{SEM}$ \textup[respectively $\mathbb{EPI}$\textup] then $\mathbf{V=M\vee N}$ where $\mathbf M$ is one of the varieties $\mathbf T$ or $\mathbf{SL}$, while $\mathbf N$ is a nil-variety.\qed
\end{proposition}

The ``semigroup part'' of this proposition was proved (in slightly weaker form and in some other terminology) in~\cite[Proposition~1.6]{Jezek-McKenzie-93}. A deduction of Proposition~\ref{mod nec} from~\cite[Proposition~1.6]{Jezek-McKenzie-93} was given explicitly in~\cite[Proposition~2.1]{Vernikov-07}. A direct and transparent proof of the ``semigroup half'' of Proposition~\ref{mod nec} not depending on a technique from~\cite{Jezek-McKenzie-93} is provided in~\cite{Shaprynskii-12}. The ``epigroup part'' of Proposition~\ref{mod nec} is a weaker version of~\cite[Theorem~1.2]{Shaprynskii-Skokov-Vernikov-16}.

Since any cancellable element of a lattice is modular, the conclusion of Proposition~\ref{mod nec} remains true whenever $\mathbf V$ is a cancellable element of $\mathbb{SEM}$ or $\mathbb{EPI}$.

\subsection{Identities of certain varieties and classes of varieties semigroups and epigroups}
\label{some identities}

For convenience of references, we formulate the following well-known fact as a lemma.

\begin{lemma}
\label{identities in SL}
An identity $\mathbf{u\approx v}$ holds in the variety $\mathbf{SL}$ if and only if $\con(\mathbf u)=\con(\mathbf v)$.\qed
\end{lemma}

The first claim of the following statement can be easily deduced from~\cite[Lemma~1]{Sapir-Sukhanov-81}, while the second claim is a simple observation.{\sloppy

}

\begin{lemma}
\label{split}
Let $\mathbf N$ be a nil-variety of semigroups.
\begin{itemize}
\item[\textup{(i)}] If $\mathbf N$ satisfies a non-trivial identity of the form $x_1x_2\cdots x_n\approx\mathbf v$ then either this identity is permutational or $\mathbf N$ satisfies the identity
\begin{equation}
\label{nilpotence}
x_1x_2\cdots x_n\approx 0.
\end{equation}
\item[\textup{(ii)}] If $\mathbf N$ satisfies a non-trivial identity of the form $x^m\approx\mathbf v$ then $\mathbf N$ satisfies also the identity $x^m\approx 0$.\qed
\end{itemize}
\end{lemma}

The following lemma is well known and can be easily checked (see~\cite{Shevrin-94,Shevrin-05}, for instance).

\begin{lemma}
\label{x^omega}
If $S$ is an arbitrary epigroup and $x\in S$ then $x\,\overline x\,=x^\omega$.\qed
\end{lemma}

The following assertion can be easily checked.

\begin{lemma}
\label{x*=0}
The identity $\overline x\,\approx 0$ holds in an epigroup $S$ if and only if $S$ is a nil-semigroup.\qed
\end{lemma}

\begin{lemma}[\!\!{\mdseries{\cite[Lemma~2.4]{Gusev-Vernikov-16}}}]
\label{unary identity}
If $\mathbf w$ is a non-semigroup word and $\con(\mathbf w)=\{x\}$ then an arbitrary epigroup variety satisfies the identity $\mathbf w\approx x^p\,\overline x\,^q$ for some integer $p\ge 0$ and some natural $q$.\qed
\end{lemma}

\section{The proof of Theorems \ref{main semigroup} and \ref{main epigroup}: the ``only if'' part}
\label{only if part}

Here we need several auxiliary facts. A word $\mathbf u$ is called \emph{linear} if every letter occurs in $\mathbf u$ at most once. If $\mathbf w$ is a word, $\con(\mathbf w)=\{x_1,x_2,\dots,x_n\}$ and $\xi\in S_n$ then we denote by $\xi[\mathbf w]$ the word obtained from $\mathbf w$ by the substitution $x_i\mapsto x_{i\xi}$ for all $i=1,2,\dots,n$. If $\mathbf V$ is a variety of semigroups or epigroups and $n$ a natural number then we denote by $\Perm_n(\mathbf V)$ the set of all permutations $\pi\in S_n$ such that $\mathbf V$ satisfies the identity $p_n[\pi]$. Clearly, $\Perm_n(\mathbf V)$ is a subgroup of $S_n$.

\begin{proposition}
\label{cancellable from V to Perm_n(V)}
If a variety of semigroups \textup[epigroups\textup] $\mathbf V$ is a cancellable element of the lattice $\mathbb{SEM}$ \textup[respectively $\mathbb{EPI}$\textup] and $n$ is a positive integer then the group $\Perm_n(\mathbf V)$ is a cancellable element of the lattice $\Sub(S_n)$.
\end{proposition}

\begin{proof}
Suppose at first that $\mathbf V$ is a semigroup variety. Clearly, the variety $\mathbf V$ is a modular element of $\mathbb{SEM}$. By Proposition~\ref{mod nec}, either $\mathbf{V=SEM}$ or $\mathbf{V=M\vee N}$ where $\mathbf M$ is one of the varieties $\mathbf T$ or $\mathbf{SL}$, while $\mathbf N$ is a nil-variety. It is evident that $\Perm_n(\mathbf{SEM})=T$ and $\Perm_n(\mathbf{SL\vee N})=\Perm_n(\mathbf N)$ for any $n$. Since $T$ is a cancellable element of $\Sub(S_n)$, we can assume that $\mathbf V=\mathbf N$. 

Put $V=\Perm_n(\mathbf V)$ for brevity. Suppose that $V$ is not a cancellable element of the lattice $\Sub(S_n)$. Then there are subgroups $X_1$ and $X_2$ of the group $S_n$ such that $V\vee X_1=V\vee X_2$ and $V\wedge X_1=V\wedge X_2$ but $X_1\ne X_2$. For $i=1,2$, we denote by $\mathbf X_i$ the variety given by the identity $x_1x_2\cdots x_{n+1}\approx 0$, all identities of the form $\mathbf w\approx 0$ where $\mathbf w$ is a word of length $n$ depending on $<n$ letters and all identities of the form $p_n[\pi]$ where $\pi\in X_i$. It is clear that $\mathbf X_1\ne\mathbf X_2$. 

It is easy to see that if $L$ is a lattice, $x,y,z\in L$, $x\vee y=x\vee z$, $x\wedge y=x\wedge z$ and one of the elements $x$, $y$ or $z$ is the largest element of $L$ then $y=z$. On the other hand, if one of the varieties $\mathbf V$, $\mathbf X_1$ or $\mathbf X_2$ satisfies the identity~\eqref{nilpotence} then one of the groups $V$, $X_1$ or $X_2$ coincides with $S_n$. Hence the identity~\eqref{nilpotence} fails in the varieties $\mathbf V$, $\mathbf X_1$ and $\mathbf X_2$.

Let now $\mathbf{u\approx v}$ be an arbitrary identity that holds in $\mathbf{V\vee X}_1$. We are going to verify that this identity holds in $\mathbf{V\vee X}_2$. Since $\mathbf{u\approx v}$ holds in $\mathbf V$, it suffices to verify that it holds in $\mathbf X_2$. The identity $\mathbf{u\approx v}$ holds in $\mathbf V$ and $\mathbf X_1$. If $\ell(\mathbf u),\ell(\mathbf v)\ge n+1$ then $\mathbf X_2$ satisfies $\mathbf u\approx 0$ and $\mathbf v\approx 0$, and thus $\mathbf{u\approx v}$. Thus, we can assume without loss of generality that $\ell(\mathbf u)\le n$. On the other hand, the definition of the variety $\mathbf X_1$ and the fact that $\mathbf{u\approx v}$ holds in $\mathbf X_1$ imply that $\ell(\mathbf u),\ell(\mathbf v)\ge n$. Therefore, $\ell(\mathbf u)=n$.

Suppose that the word $\mathbf u$ is linear. By Lemma~\ref{split}(i), either the identity $\mathbf{u\approx v}$ is permutational or the variety $\mathbf{V\vee X}_1$ satisfies the identity $\mathbf u\approx 0$. However, we have proved above that the second case is impossible. Therefore, $\mathbf{u\approx v}$ is an identity of the form $p_n[\pi]$. Since it holds both in $\mathbf V$ and $\mathbf X_1$, we have $\pi\in V\wedge X_1=V\wedge X_2$. Hence $\pi\in X_2$, and therefore the identity $\mathbf{u\approx v}$ holds in $\mathbf X_2$. 

It remains to consider the case when the word $\mathbf u$ is non-linear. Then it depends on $<n$ letters. Therefore, $\mathbf u\approx 0$ holds in the varieties $\mathbf X_1$ and $\mathbf X_2$. If $\ell(\mathbf v)>n$ or $\mathbf v$ is a word of length $n$ depending on $<n$ letters then $\mathbf v\approx 0$ holds in $\mathbf X_2$, whence $\mathbf{u\approx v}$ holds in $\mathbf X_2$. Finally, if $\ell(\mathbf v)=n$ and $\mathbf v$ depends on $n$ letters then the word $\mathbf v$ is linear and we can complete our considerations by the same arguments as in the previous paragraph.

Thus, if the identity $\mathbf{u\approx v}$ holds in $\mathbf{V\vee X}_1$ then it holds in $\mathbf{V\vee X}_2$ too. This means that $\mathbf{V\vee X}_2\subseteq\mathbf{V\vee X}_1$. The reverse inclusion can be verified analogously, whence $\mathbf{V\vee X}_1=\mathbf{V\vee X}_2$.

Let now $\mathbf{u\approx v}$ be an arbitrary identity that holds in the variety $\mathbf{V\wedge X}_1$. We aim to verify that it holds in $\mathbf{V\wedge X}_2$. Let the sequence of words
\begin{equation}
\label{deduction}
\mathbf{u=u}_0,\mathbf u_1,\dots,\mathbf u_k=\mathbf v
\end{equation}
be a deduction of the identity $\mathbf{u\approx v}$ from the identities of the varieties $\mathbf V$ and $\mathbf X_1$. This means that, for any $i=0,1,\dots,k-1$, the identity $\mathbf u_i\approx\mathbf u_{i+1}$ holds in one of the varieties $\mathbf V$ or $\mathbf X_1$. 

Suppose that there is an index $i$ such that $\mathbf u_i$ is a linear word of length $n$. If $i>0$ then Lemma~\ref{split}(i) implies that either $\mathbf u_{i-1}$ is a linear word and $\con(\mathbf u_{i-1})=\con(\mathbf u_i)$ or one of the varieties $\mathbf V$ or $\mathbf X_1$ satisfies the identity~\eqref{nilpotence}. Analogously, if $i<k$ then either $\mathbf u_{i+1}$ is a linear word and $\con(\mathbf u_i)=\con(\mathbf u_{i+1})$ or one of the varieties $\mathbf V$ and $\mathbf X_1$ satisfies the identity~\eqref{nilpotence}. As we have seen above, the latter is impossible. Therefore, the words adjacent to $\mathbf u_i$ in the sequence~\eqref{deduction} are linear words of length $n$ depending on the same letters as $\mathbf u_i$. By induction, this means that all the words $\mathbf u_0$, $\mathbf u_1$, \dots, $\mathbf u_k$ are linear words of length $n$ depending on the same letters. We can assume without loss of generality that $\con(\mathbf u)=\{x_1,x_2,\dots,x_n\}$. There are permutations $\pi_0,\pi_1,\dots,\pi_{k-1}\in S_n$ such that $\mathbf u_i=\pi_i[\mathbf u_{i-1}]$ for each $i=1,2,\dots,k$. Clearly, $\pi_i\in V$ [respectively $\pi_i\in X_1$] whenever the identity $\mathbf u_{i-1}\approx\mathbf u_i$ holds in the variety $\mathbf V$ [respectively $\mathbf X_1$]. Put $\pi=\pi_0\pi_1\cdots\pi_{k-1}$. Then $\pi\in V\vee X_1=V\vee X_2$. Therefore, there are permutations $\sigma_0,\sigma_1,\dots,\sigma_{m-1}\in S_n$ such that $\pi=\sigma_0\sigma_1\cdots\sigma_{m-1}$ and, for all $i=0,1,\dots,m-1$, the permutation $\sigma_i$ lies in either $V$ or $X_2$. Put $\mathbf v_0=\mathbf u$ and $\mathbf v_i=\sigma_i[\mathbf v_{i-1}]$ for each $i=1,2,\dots,m$. Obviously, $\mathbf v_m=\mathbf v$ and, for any $i=0,1,\dots,m-1$, the identity $\mathbf v_i\approx\mathbf v_{i+1}$ holds in one of the varieties $\mathbf V$ or $\mathbf X_2$. Thus, the sequence of words
$$
\mathbf{u=v}_0,\mathbf v_1,\dots,\mathbf v_m=\mathbf v
$$
is a deduction of the identity $\mathbf{u\approx v}$ from the identities of the varieties $\mathbf V$ and $\mathbf X_2$. Therefore, the identity $\mathbf{u\approx v}$ holds in the variety $\mathbf{V\wedge X}_2$.

Suppose now that there are no linear words of length $n$ among the words $\mathbf u_0$, $\mathbf u_1$, \dots, $\mathbf u_k$. The definition of the variety $\mathbf X_1$ shows that if this variety satisfies a non-trivial identity $\mathbf{p\approx q}$ then $\ell(\mathbf p),\ell(\mathbf q)\ge n$. This means that:
\begin{itemize}
\item[$\bullet$] $\ell(\mathbf u_i)\ge n$ for all $i=1,2,\dots,k-1$;
\item[$\bullet$] either $\ell(\mathbf u_0)\ge n$ or the identity $\mathbf u_0\approx\mathbf u_1$ holds in $\mathbf V$;
\item[$\bullet$] either $\ell(\mathbf u_k)\ge n$ or the identity $\mathbf u_{k-1}\approx\mathbf u_k$ holds in $\mathbf V$.
\end{itemize}
Let $1\le i\le k-1$. In view of what we said above, either $\ell(\mathbf u_i)>n$ or $\ell(\mathbf u_i)=n$ and $\mathbf u_i$ depends on $<n$ letters. In both cases the identity $\mathbf u_i\approx 0$ holds in $\mathbf X_2$. Hence $\mathbf X_2$ satisfies the identities $\mathbf u_1\approx 0\approx\mathbf u_{k-1}$. Moreover, if $\mathbf u_0\approx\mathbf u_1$ or $\mathbf u_{k-1}\approx\mathbf u_k$ does not hold in $\mathbf V$ then it holds in $\mathbf X_2$. Therefore, the sequence of words $\mathbf{u=u}_0,\mathbf u_1,\mathbf u_{k-1},\mathbf u_k=\mathbf v$ is a deduction of the identity $\mathbf{u\approx v}$ from the identities of the varieties $\mathbf V$ and $\mathbf X_2$. Thus, we have again that the identity $\mathbf{u\approx v}$ holds in the variety $\mathbf{V\wedge X}_2$.

We have proved that if the identity $\mathbf{u\approx v}$ holds in $\mathbf{V\wedge X}_1$ then it holds in $\mathbf{V\wedge X}_2$. Therefore, $\mathbf{V\wedge X}_2\subseteq\mathbf{V\wedge X}_1$. The reverse inclusion can be verified analogously, whence $\mathbf{V\wedge X}_1=\mathbf{V\wedge X}_2$.

Thus, $\mathbf{V\vee X}_1=\mathbf{V\vee X}_2$ and $\mathbf{V\wedge X}_1=\mathbf{V\wedge X}_2$. Since $\mathbf V$ is a cancellable element of $\mathbb{SEM}$, this implies that $\mathbf X_1=\mathbf X_2$. However, this contradicts the choice of the groups $X_1$ and $X_2$.

We obtain the desirable conclusion in the semigroup case. Suppose now that $\mathbf V$ is an epigroup variety. By Proposition~\ref{mod nec}, $\mathbf{V=M\vee N}$ where $\mathbf M$ is one of the varieties $\mathbf T$ or $\mathbf{SL}$, while $\mathbf N$ is a nil-variety. As in the semigroup case, we can assume that $\mathbf V=\mathbf N$. Note that the varieties $\mathbf X_1$ and $\mathbf X_2$ considered above in this proof are periodic. Therefore, they can be considered as epigroup varieties. Now we can complete considerations in the epigroup case by literally repeating arguments given above in the semigroup one.
\end{proof}

Proposition~\ref{cancellable from V to Perm_n(V)} and Lemma~\ref{subgroups cancellable} immediately imply the following 

\begin{proposition}
\label{alternative for cancellable}
Let a variety of semigroups \textup[epigroups\textup] $\mathbf V$ be a cancellable element of the lattice $\mathbb{SEM}$ \textup[respectively $\mathbb{EPI}$\textup] and $n$ be a natural number. If $\mathbf V$ satisfies some permutational identity of length $n$ then it satisfies all such identities.\qed
\end{proposition}

Let $F$ be the free semigroup over a countably infinite alphabet and $F^1$ be the semigroup $F$ with the new identity element adjoined. We treat this identity element as the empty word. We denote by $\Aut(F)$ and $\End(F)$ the group of automorphisms and the monoid of endomorphisms on $F$ respectively. Let $\mathbf u,\mathbf v\in F$. We write $\mathbf{u\le v}$ if $\mathbf{v=a}\xi(\mathbf u)\mathbf b$ for some $\xi\in\End(F)$ and some $\mathbf a,\mathbf b\in F^1$. If $\mathbf{u\le v}$ and $\mathbf{u\ne v}$ then we write $\mathbf{u<v}$. We say that words $\mathbf u$ and $\mathbf v$ are \emph{incomparable} if $\mathbf{u\not\le v}$ and $\mathbf{v\not\le u}$. We will say that words $\mathbf p$ and $\mathbf q$ are \emph{equivalent} if $\mathbf{p\le q}$ and $\mathbf{q\le p}$. It is clear that the words $\mathbf p$ and $\mathbf q$ are equivalent if and only if $\mathbf q=\varphi(\mathbf p)$ for some $\varphi\in\Aut(F)$.

\begin{lemma}
\label{if u=v=0 then w=0}
Suppose that $\mathbf u$, $\mathbf v$ and $\mathbf w$ are pairwise incomparable words and $\con(\mathbf u)=\con(\mathbf v)=\con(\mathbf w)$. If a semigroup variety $\mathbf V$ is a cancellable element of the lattice $\mathbb{SEM}$ and $\mathbf V$ satisfies the identities $\mathbf u\approx 0$ and $\mathbf v\approx 0$ then it satisfies the identity $\mathbf w\approx 0$ too.{\sloppy

}
\end{lemma}

\begin{proof}
Consider the set 
$$
I=\{\mathbf s\in F\mid\mathbf{u<s}\text{ or }\mathbf{v<s}\text{ or }\mathbf{w<s}\}.
$$ 
Note that $\mathbf u,\mathbf v,\mathbf w\notin I$ because these words are pairwise incomparable. Therefore, $\varphi(\mathbf u),\varphi(\mathbf v),\varphi(\mathbf w)\notin I$ for any $\varphi\in\Aut(F)$. Obviously, $I$ is a fully invariant ideal of $F$. Put
$$
\mathbf N=\var\{\mathbf s\approx 0\mid\mathbf s\in I\},\ \mathbf U=\mathbf N\wedge\var\{\mathbf{u\approx w}\}\text{ and }\mathbf W=\mathbf N\wedge\var\{\mathbf{v\approx w}\}.
$$

Now we interrupt the proof of Lemma~\ref{if u=v=0 then w=0} in order to prove the following

\begin{lemma}
\label{eq theory of U}
Non-trivial identities that hold in $\mathbf U$ are only the identities $\mathbf{s\approx t}$ with $\mathbf s,\mathbf t\in I$ and the identities $\varphi(\mathbf u)\approx\varphi(\mathbf w)$ or $\varphi(\mathbf w)\approx\varphi(\mathbf u)$ where $\varphi\in\Aut(F)$. Analogously, non-trivial identities that hold in $\mathbf W$ are only the identities $\mathbf{s\approx t}$ with $\mathbf s,\mathbf t\in I$ and the identities $\mathbf{\varphi(v)\approx\varphi(w)}$ or $\mathbf{\varphi(w)\approx\varphi(v)}$ where $\varphi\in\Aut(F)$.
\end{lemma}

\begin{proof}
By symmetry, it suffices to verify the first claim of the lemma. To do this, we will first describe identities $\mathbf{a\approx b}$ that directly follow from the identity system defining $\mathbf U$. Any such identity directly follows either from an identity $\mathbf s\approx 0$ where $\mathbf s\in I$ or from the identity $\mathbf{u\approx w}$. In the first case $\mathbf a,\mathbf b\in I$ because $I$ is a fully invariant ideal. In the second case we have either $\mathbf{a=c}\varphi(\mathbf u)\mathbf d$ and $\mathbf{b=c}\varphi(\mathbf w)\mathbf d$ or vice versa where $\mathbf c,\mathbf d\in F^1$ and $\varphi\in\End(F)$. If at least one of the words $\mathbf c$ and $\mathbf d$ is non-empty or $\varphi$ does not act on $\con(\mathbf u)$ as an automorphism of $F$ then $\mathbf{u<a}$ and $\mathbf{w<b}$ or vice versa (here we use the equality $\con(\mathbf u)=\con(\mathbf w)$). Hence $\mathbf a,\mathbf b\in I$ again. Finally, if $\mathbf c$ and $\mathbf d$ are empty and $\varphi$ acts on $\con(\mathbf u)$ as an automorphism of $F$ then we can suppose that $\varphi\in\Aut(F)$. 

In order to describe all non-trivial identities of $\mathbf U$, we consider a deduction of an identity $\mathbf{p\approx q}$ that holds in $\mathbf U$ from the basis of identities of $\mathbf U$:
$$
\mathbf{p=w}_0,\mathbf w_1,\dots,\mathbf w_k=\mathbf q.
$$
If $\mathbf w_0\in I$ then the identity $\mathbf w_0\approx\mathbf w_1$ does not have the form $\varphi(\mathbf u)\approx\varphi(\mathbf w)$ or $\varphi(\mathbf w)\approx\varphi(\mathbf u)$ where $\varphi\in\Aut(F)$ because $\varphi(\mathbf u),\varphi(\mathbf w)\notin I$. Hence $\mathbf w_1\in I$. Now a simple induction shows that $\mathbf w_2,\dots,\mathbf w_k\in I$. If $\mathbf w_0=\varphi(\mathbf u)$ then $\mathbf w_0\notin I$. Therefore, the identity $\mathbf w_0\approx\mathbf w_1$ has either the form $\psi(\mathbf u)\approx\psi(\mathbf w)$ or the form $\psi(\mathbf w)\approx\psi(\mathbf u)$ where $\psi\in\Aut(F)$. The latter is impossible because $\mathbf u$ and $\mathbf w$ are incomparable, whence $\varphi(\mathbf u)=\psi(\mathbf u)$ and $\varphi(\mathbf w)=\psi(\mathbf w)$. Hence the restriction of $\varphi$ on the set $\con(\mathbf u)$ coincides with the restriction of $\psi$ on this set. Since $\con(\mathbf u)=\con(\mathbf w)$, we can suppose that $\varphi=\psi$. Now a simple induction shows that $\mathbf w_2=\varphi(\mathbf u)$, $\mathbf w_3=\varphi(\mathbf w)$, \dots, so the identity $\mathbf{p\approx q}$ either is trivial or has the form $\varphi(\mathbf u)\approx\varphi(\mathbf w)$. Similar arguments show that if $\mathbf w_0=\varphi(\mathbf w)$ then the identity $\mathbf{p\approx q}$ either is trivial or has the form $\varphi(\mathbf w)\approx\varphi(\mathbf u)$.
\end{proof}

Let us return to the proof of Lemma~\ref{if u=v=0 then w=0}. The variety $\mathbf V\wedge\mathbf U$ satisfies the identities $\mathbf{v\approx u\approx w}$. Therefore, $\mathbf{V\wedge U\subseteq W}$, whence $\mathbf{V\wedge U\subseteq V\wedge W}$. Similar arguments show that $\mathbf{V\wedge W\subseteq V\wedge U}$, whence $\mathbf{V\wedge U=V\wedge W}$.

Suppose that $\mathbf V$ does not satisfy the identity $\mathbf w\approx 0$. Hence it satisfies none of the identities $\mathbf{u\approx w}$ and $\mathbf{v\approx w}$, and therefore none of the identities $\varphi(\mathbf u)\approx\varphi(\mathbf w)$ and $\varphi(\mathbf v)\approx\varphi(\mathbf w)$ where $\varphi\in\Aut(F)$. Let us consider a non-trivial identity $\mathbf{a\approx b}$ which holds in $\mathbf{V\vee U}$. By Lemma~\ref{eq theory of U}, $\mathbf a,\mathbf b\in I$. We see that $\mathbf{N\subseteq V\vee U}$, whence $\mathbf{V\vee N\subseteq V\vee U}$. On the other hand, $\mathbf{V\vee U\subseteq V\vee N}$ because $\mathbf{U\subseteq N}$. Therefore, $\mathbf{V\vee N=V\vee U}$. Similar arguments show that $\mathbf{V\vee N=V\vee W}$, whence $\mathbf{V\vee U=V\vee W}$. Since the variety $\mathbf V$ is a cancellable element of $\mathbb{SEM}$, we have that $\mathbf{U=W}$. However, this contradicts Lemma~\ref{eq theory of U}.
\end{proof}

\begin{lemma}
\label{incomparable words}
Non-equivalent words with equal length and equal content are incomparable.
\end{lemma}

\begin{proof} 
Suppose that $\mathbf a$ and $\mathbf b$ are non-equivalent words with equal length and equal content and $\mathbf{a=c}\varphi(\mathbf b)\mathbf d$ where $\mathbf c,\mathbf d\in F^1$ and $\varphi\in\End(F)$. We have $\ell(\mathbf a)=\ell(\mathbf b)\le\ell(\mathbf c\varphi(\mathbf b)\mathbf d)=\ell(\mathbf a)$ which is possible only when $\mathbf c$ and $\mathbf d$ are empty and $\varphi$ maps each letter from $\con(\mathbf b)$ to a letter. Furthermore, $|\con(\mathbf a)|=|\con(\mathbf b)|\ge|\con(\varphi(\mathbf b))|=|\con(\mathbf a)|$, whence $\varphi$ is one-to-one on $\con(\mathbf a)$. Hence we can suppose that $\varphi\in\Aut(F)$.
\end{proof}

\begin{lemma}
\label{3 0-reduced identities}
If a nil-variety of semigroups $\mathbf V$ is a cancellable element of the lattice $\mathbb{SEM}$ then it satisfies the identities 
\begin{equation}
\label{xxy=xyx=yxx=0}
x^2y\approx xyx\approx yx^2\approx 0.
\end{equation}
\end{lemma}

\begin{proof}
Being a nil-variety, $\mathbf V$ satisfies identities $x^ny\approx yx^n\approx 0$ for some $n$. Now Lemma~\ref{incomparable words} is applied with the conclusion that the words $x^ny$, $yx^n$ and $x^{n-k}yx^k$ for all $k=1,2,\dots,n-1$ are pairwise incomparable. Then Lemma~\ref{if u=v=0 then w=0} implies that the variety $\mathbf V$ satisfies the identities $x^{n-k}yx^k\approx 0$ for all $k=1,2,\dots,n-1$. Suppose that $n\ge 4$. One can consider the words $x^{n-2}yx^2$, $x^{n-3}yx^3$ and $x^{n-1}y$. We have $x^{n-2}yx^2\not\le x^{n-1}y$ and $x^{n-3}yx^3\not\le x^{n-1}y$ because $\ell(x^{n-2}yx^2)=\ell(x^{n-3}yx^3)>\ell(x^{n-1}y)$. The words $x^{n-2}yx^2$ and $x^{n-3}yx^3$ are incomparable by Lemma~\ref{incomparable words}. Finally, $x^{n-1}y\not\le x^{n-2}yx^2$ and $x^{n-1}y\not\le x^{n-3}yx^3$ because the words $x^{n-2}yx^2$ and $x^{n-3}yx^3$ do not contain any \mbox{$(n-1)$-th} powers which are not their suffixes (note that $x^{n-3}yx^3$ contains the suffix $x^3$ which is an \mbox{$(n-1)$-th} power if $n=4$). Now we can apply Lemma~\ref{if u=v=0 then w=0} and conclude that $\mathbf V$ satisfies the identity $x^{n-1}y\approx 0$. By the dual arguments, $\mathbf V$ satisfies $yx^{n-1}\approx 0$ as well. Now a simple induction shows that $\mathbf V$ satisfies the identities $x^{n-2}y\approx yx^{n-2}\approx 0$, $x^{n-3}y\approx yx^{n-3}\approx 0$, \dots, $x^3y\approx yx^3\approx 0$. As we have observed in the beginning of this paragraph, the identities $x^ny\approx yx^n\approx 0$ imply in $\mathbf V$ the identities $x^{n-k}yx^k\approx 0$ for all $k=1,2,\dots,n-1$. Now we apply this observation with $n=3$ and $k=1,2$ with the conclusion that $\mathbf V$ satisfies the identities $x^2yx\approx xyx^2\approx 0$.
 
Lemma~\ref{incomparable words} implies that the words $x^3y$, $yx^3$ and $(xy)^2$ are pairwise incomparable. Then $\mathbf V$ satisfies the identity $(xy)^2\approx 0$ by Lemma~\ref{if u=v=0 then w=0}. Now we consider the words $xyx^2$, $(xy)^2$ and $x^2y$. Since $\ell(xyx^2)=\ell((xy)^2)>\ell(x^2y)$, we conclude that $xyx^2\not\le x^2y$ and $(xy)^2\not\le x^2y$. The words $xyx^2$ and $(xy)^2$ are incomparable by Lemma~\ref{incomparable words}. Finally, we have $x^2y\not\le xyx^2$ and $x^2y\not\le(xy)^2$ because the words $xyx^2$ and $(xy)^2$ do not contain any squares that are not their suffixes. By Lemma~\ref{if u=v=0 then w=0}, $\mathbf V$ satisfies the identity $x^2y\approx 0$. By the dual arguments, $\mathbf V$ satisfies also the identity $yx^2\approx 0$. The words $x^2y$, $yx^2$ and $xyx$ are incomparable by Lemma~\ref{incomparable words}. It remains to apply Lemma~\ref{if u=v=0 then w=0} and conclude that $\mathbf V$ satisfies the identity $xyx\approx 0$. 
\end{proof}

Now we are well prepared to prove necessity in Theorems~\ref{main semigroup} and~\ref{main epigroup}.

\smallskip 

\emph{Proof of Theorem}~\ref{main semigroup}. \emph{Necessity}. Let $\mathbf V$ be a proper semigroup variety that is a cancellable element of the lattice $\mathbb{SEM}$. Since each cancellable element is modular, Proposition~\ref{mod nec} implies that $\mathbf{V=M\vee N}$ where $\mathbf M\in\{\mathbf T,\mathbf{SL}\}$ and $\mathbf N$ is a nil-variety. In view of Corollary~\ref{join with SL}, we can assume that $\mathbf{V=N}$. We need to verify that $\mathbf{V=T}$ or $\mathbf{V=X}_{m,n}$ or $\mathbf{V=Y}_{m,n}$ for some $2\le m\le n\le\infty$.

By Lemma~\ref{3 0-reduced identities}, the variety $\mathbf V$ satisfies the identities~\eqref{xxy=xyx=yxx=0}. If each identity that holds in $\mathbf V$ follows from~\eqref{xxy=xyx=yxx=0} then $\mathbf{V=X}_{\infty,\infty}$. Suppose that $\mathbf V$ satisfies an identity $\mathbf{u\approx v}$ that does not follow from the identities~\eqref{xxy=xyx=yxx=0}. These identities imply any identities of the form $\mathbf p\approx 0\approx\mathbf q$ such that each of the words $\mathbf p$ and $\mathbf q$ is non-linear and is not the square of a letter. Hence we can assume without loss of generality that either $\mathbf u=x_1x_2\cdots x_n$ for some $n$ or $\mathbf u=x^2$. Now we can apply Lemma~\ref{split} and conclude that $\mathbf V$ satisfies either the identity $x^2\approx 0$ or the identity~\eqref{nilpotence} or some permutational identity of length $n$. In the last case $\mathbf V$ satisfies all permutational identities of length $n$ by Proposition~\ref{alternative for cancellable}. Thus, $\mathbf V$ is given within $\mathbf X_{\infty,\infty}$ either by the identity~\eqref{nilpotence} for some natural $n$ or by the idenity $x^2\approx 0$ or by the identity system
\begin{equation}
\label{all permutations}
\{p_m[\sigma]\mid\sigma\in S_m\}
\end{equation}
for some natural $m$ or by a combination of the listed idenities and identity system. Clearly, the identity~\eqref{nilpotence} implies the system~\eqref{all permutations} whenever $m\ge n$. Evidently, all the saying above is equivalent to the desirable conclusion.

Necessity of Theorem~\ref{main semigroup} is proved. The following observation will be used in the proof of necessity of Theorem~\ref{main epigroup}. As we have mentioned in Section~\ref{intr}, the lattice $\mathbb{PER}$ of all periodic semigroup varieties is a sublattice in both the lattices $\mathbb{SEM}$ and $\mathbb{EPI}$. Note that the varieties $\mathbf U$ and $\mathbf W$ that appear in the proofs of Lemmas~\ref{if u=v=0 then w=0} and~\ref{eq theory of U} are nil-varieties and therefore, are periodic. Hence the proof of necessity of Theorem~\ref{main semigroup} implies the following

\begin{corollary}
\label{nec epi nil}
If a periodic semigroup variety $\mathbf V$ is a cancellable element of the lattice $\mathbb{PER}$ then $\mathbf{V=M\vee N}$ where $\mathbf M$ is one of the varieties $\mathbf T$ or $\mathbf{SL}$, while $\mathbf N$ is one of the varieties $\mathbf T$, $\mathbf X_{m,n}$ or $\mathbf Y_{m,n}$ with $2\le m\le n\le\infty$.\qed
\end{corollary}

\emph{Proof of Theorem}~\ref{main epigroup}. \emph{Necessity}. Let an epigroup variety $\mathbf V$ be a cancellable element of the lattice $\mathbb{EPI}$. Clearly, $\mathbf V$ is a modular element of $\mathbb{EPI}$. Now Proposition~\ref{mod nec} is applied with the conclusion that $\mathbf{V=M\vee N}$ where $\mathbf M$ is one of the varieties $\mathbf T$ or $\mathbf{SL}$, while $\mathbf N$ is a nil-variety. Then the variety $\mathbf V$ is periodic, whence it can be considered as a semigroup variety. Clearly, $\mathbf V$ is a cancellable element of the lattice $\mathbb{PER}$. It remains to refer to Corollary~\ref{nec epi nil}.

\section{The proof of Theorems \ref{main semigroup} and \ref{main epigroup}: the ``if'' part}
\label{if part}

First of all, we note that the known results easily imply that the varieties $\mathbf X_{n,n}$ and $\mathbf Y_{n,n}$ with $2\le n\le\infty$ are cancellable elements of both the lattices $\mathbb{SEM}$ and $\mathbb{EPI}$. Indeed, these varieties are distributive elements of the lattice $\mathbb{SEM}$ by~\cite[Theorem~1.1]{Vernikov-Shaprynskii-10} and modular elements of this lattice by~\cite[Corollary~1.2]{Vernikov-Shaprynskii-10}. It is well known that a distributive and modular element of a lattice is a standard element (see~\cite[Lemma~II.1.1]{Gratzer-Schmidt-61}, for instance). Since a standard element of a lattice is a cancellable one, we are done in the semigroup case. In the epigroup case it suffices to refer to the fact that the mentioned varieties are standard elements of $\mathbb{EPI}$ by~\cite[Theorem~1.1 and Corollary~1.2]{Skokov-15}. But the proof given below embraces all varieties mentioned in Theorems~\ref{main semigroup} and~\ref{main epigroup}, including the varieties $\mathbf X_{n,n}$ and $\mathbf Y_{n,n}$.

It is convenient for us to start the proof of sufficiency with Theorem~\ref{main epigroup}.

\smallskip

\emph{Proof of Theorem}~\ref{main epigroup}. \emph{Sufficiency}. Suppose that an epigroup variety $\mathbf V$ has the form indicated in the formulation of Theorem~\ref{main epigroup}. We need to verify that $\mathbf V$ is a cancellable element of $\mathbb{EPI}$. Let $\mathbf U$ and $\mathbf W$ be epigroup varieties such that $\mathbf{V\vee U=V\vee W}$ and $\mathbf{V\wedge U=V\wedge W}$. We need to check that $\mathbf{U=W}$. Corollary~\ref{join with SL} allows us to suppose that $\mathbf V$ is one of the varieties $\mathbf T$, $\mathbf X_{m,n}$ or $\mathbf Y_{m,n}$ with $2\le m\le n\le\infty$. Thus, $\mathbf V$ is a nil-variety. In particular, it is periodic. The case when $\mathbf{V=T}$ is evident. By symmetry, it suffices to show that a non-trivial identity $\mathbf{u\approx v}$ holds in $\mathbf W$ whenever it holds in $\mathbf U$. So, let $\mathbf U$ satisfy the identity $\mathbf{u\approx v}$. Note that if this identity holds in $\mathbf V$ then it holds in $\mathbf{V\vee U=V\vee W}$ and therefore, in $\mathbf W$.

It is evident that the varieties $\mathbf U$ and $\mathbf W$ are either both periodic or both non-periodic, because otherwise one of the varieties $\mathbf{V\vee U}$ or $\mathbf{V\vee W}$ is periodic and the other one is not, contradicting the equality $\mathbf{V\vee U=V\vee W}$. By Corollary~\ref{join with SL}, we can assume that $\mathbf U,\mathbf{W\supseteq SL}$. Lemma~\ref{identities in SL} implies now that $\con(\mathbf u)=\con(\mathbf v)$. Further considerations are divided into two cases.

\smallskip

\emph{Case}~1: $|\con(\mathbf u)|=1$. Here there are three subcases.

\smallskip

\emph{Subcase}~1.1: $\mathbf u$ and $\mathbf v$ are semigroup words. This means that the identity $\mathbf{u\approx v}$ has the form $x^m\approx x^n$ for some different $m$ and $n$. Therefore, the variety $\mathbf U$ is periodic. Hence the variety $\mathbf{V\vee U=V\vee W}$ is periodic too, and so the variety $\mathbf W$ has the same property. Thus, we can suppose that all varieties under consideration are semigroup varieties. Let $F_1$ be the free cyclic semigroup. We denote by $\alpha$, $\beta$ and $\gamma$ the fully invariant congruences on $F_1$ corresponding to $\mathbf V$, $\mathbf U$ and $\mathbf W$ respectively. The equalities $\mathbf{V\vee U=V\vee W}$ and $\mathbf{V\wedge U=V\wedge W}$ imply that $\alpha\wedge\beta=\alpha\wedge\gamma$ and $\alpha\vee\beta=\alpha\vee\gamma$. It is well known that the congruence lattice of an arbitrary cyclic semigroup is distributive (see~\cite[Theorem~2.17]{Mitsch-83}, for instance). In particular, the lattice of congruences on $F_1$ is distributive. Therefore, each of its element is cancellable. Hence $\beta=\gamma$. Since $x^m\beta x^n$, we have $x^m\gamma x^n$. Thus, the identity $x^m\approx x^n$ holds in $\mathbf W$.

\smallskip
 
\emph{Subcase}~1.2: $\mathbf u$ and $\mathbf v$ are not semigroup words. Then Lemma~\ref{x*=0} applies with the conclusion that the variety $\mathbf V$ satisfies the identities $\mathbf u\approx 0\approx\mathbf v$ and we are done.

\smallskip 

\emph{Subcase}~1.3: one of the words $\mathbf u$ and $\mathbf v$, say $\mathbf u$, is a semigroup word, while the other is not. Then $\mathbf u=x^m$ for some natural $m$. By Lemma~\ref{unary identity}, there are $p\ge 0$ and $q\in\mathbb N$ such that the variety $\mathbf U$ satisfies the identity 
\begin{equation}
\label{v=x^p(x*)^q}
\mathbf v\approx x^p\,\overline x\,^q
\end{equation}
and therefore, the identities $x^m\approx\mathbf v\approx x^p\,\overline x\,^q$. Thus, the identity
\begin{equation}
\label{x^m=x^p(x*)^q}
x^m\approx x^p\,\overline x\,^q.
\end{equation}
holds in $\mathbf U$.

We denote by $\Gr S$ the set of all group elements of an epigroup $S$. Let $S\in\mathbf U$ and $x\in S$. If $p\le q$ then Lemma~\ref{x^omega} implies that
$$
x^m=x^p\,\overline x\,^q=x^\omega\,\overline x\,^{q-p}\in\Gr S.
$$
Further, let $p>q$. It is well known and can be easily checked (see~\cite{Shevrin-94,Shevrin-05}, for instance) that an arbitrary epigroup satisfies the identities $x^\omega x\approx xx^\omega\approx\,\overline{\overline x}$\,. This fact and Lemma~\ref{x^omega} imply that
$$
x^m=x^p\,\overline x\,^q=x^{p-q}x^\omega=x^{p-q}(x^\omega)^{p-q}=(xx^\omega)^{p-q}=\bigl(\,\overline{\overline x}\,\bigr)^{p-q}\in\Gr S.
$$
So, $x^m\in\Gr S$ in any case, whence $\mathbf U$ satisfies the identity
\begin{equation}
\label{x^m=x^m x^omega}
x^m\approx x^mx^\omega.
\end{equation}

The identity $x^m\approx\mathbf v$ holds in $\mathbf{V\wedge U=V\wedge W}$. Therefore, there is a deduction of this identity from identities of the varieties $\mathbf V$ and $\mathbf W$. In particular, one of these varieties satisfies a non-trivial identity of the form $x^m\approx\mathbf w$. If this identity holds in $\mathbf V$ then Lemmas~\ref{split}(ii) and~\ref{x*=0} imply that $\mathbf V$ satisfies the identities $x^m\approx 0\approx\mathbf v$ and we are done.

It remains to consider the case when the identity $x^m\approx\mathbf w$ holds in $\mathbf W$. Suppose at first that $\mathbf w$ is a semigroup word. If $\ell(\mathbf w)=m$ then there exists a letter $y\in\con(\mathbf w)$ with $y\ne x$. Substituting $y^2$ to $y$ in the identity $x^m\approx\mathbf w$, we obtain an identity of the form $x^m\approx\mathbf w'$ with $\ell(\mathbf w')\ne m$. Thus, we can assume that $\ell(\mathbf w)\ne m$. Then equating all letters from $\con(\mathbf w)$ to $x$, we deduce from $x^m\approx\mathbf w$ an identity of the form $x^m\approx x^n$ with $m\ne n$. Thus, the variety $\mathbf W$ is periodic. Therefore, $\mathbf U$ is periodic too. Thus, all the varieties under consideration are periodic. Being periodic, the variety $\mathbf{U\vee W}$ satisfies an identity of the form $x^r\approx x^{r+s}$ for some natural $r$ and $s$. It is easy to see that this identity implies an identity of the form $\overline x\,\approx x^{(r+q)s-1}$ for any natural $q$. Clearly, $(r+q)s-1>m$ for some $q$. Therefore, $\mathbf{U\vee W}$ satisfies the identity $\mathbf v\approx x^k$ for some $k>m$. Then the identities $x^m\approx\mathbf v\approx x^k$ hold in $\mathbf U$. The arguments given in Subcase~1.1 imply that the variety $\mathbf W$ satisfies the identities $x^m\approx x^k\approx\mathbf v$.

Suppose now that $\mathbf w$ is a non-semigroup word. Considerations analogous to mentioned in the second paragraph of Subcase~1.3 allow us to check that the variety $\mathbf W$ satisfies the identity~\eqref{x^m=x^m x^omega}. Lemmas~\ref{x^omega} and~\ref{x*=0} imply that $\mathbf V$ satisfies the identities $x^m x^\omega\approx x^{m+1}\,\overline x\,\approx 0\approx x^p\,\overline x\,^q$. Further, the identities $x^mx^\omega\approx x^m\approx x^p\,\overline x\,^q$ hold in $\mathbf U$ by~\eqref{x^m=x^p(x*)^q} and~\eqref{x^m=x^m x^omega}. So, the identity
\begin{equation}
\label{x^m x^omega=x^p(x*)^q}
x^m x^\omega\approx x^p\,\overline x\,^q
\end{equation}
holds in the variety $\mathbf{V\vee U=V\vee W}$ and therefore, in $\mathbf W$.

Lemma~\ref{x*=0} implies that $\mathbf V$ satisfies the identities $\mathbf v\approx 0\approx x^p\,\overline x\,^q$. Therefore, the identity~\eqref{v=x^p(x*)^q} holds in $\mathbf{V\vee U=V\vee W}$ and therefore, in $\mathbf W$. Combining the identities~\eqref{v=x^p(x*)^q},~\eqref{x^m=x^m x^omega} and~\eqref{x^m x^omega=x^p(x*)^q}, we obtain that $\mathbf W$ satisfies the identities $x^m\approx x^m x^\omega\approx x^p\,\overline x\,^q\approx\mathbf v$ that completes the proof in Case~1.

\smallskip

\emph{Case}~2: $|\con(\mathbf u)|=k>1$. Every non-semigroup word equals~0 in $\mathbf V$ by Lemma~\ref{x*=0}. Further, every semigroup non-linear word depending on $>1$ letters also equals~0 in $\mathbf V$ because $\mathbf V$ satisfies the identities~\eqref{xxy=xyx=yxx=0}. Thus, if neither $\mathbf u$ nor $\mathbf v$ is a semigroup linear word then the identities $\mathbf u\approx 0\approx\mathbf v$ hold in $\mathbf V$ and we are done. Hence we can suppose without loss of generality that $\mathbf u= x_1x_2\cdots x_k$. Since $\con(\mathbf u)=\con(\mathbf v)$, we have $\ell(\mathbf v)\ge k$. Further considerations are divided into three subcases.

\smallskip

\emph{Subcase}~2.1: $\mathbf v$ is not a semigroup word. Using arguments from Case~2 in the proof of sufficiency of Theorem~1.1 in the article~\cite{Skokov-18}, we can prove that both the varieties $\mathbf U$ and $\mathbf W$ satisfy the identity
$$
x_1x_2\cdots x_k\approx x_1\cdots x_{i-1}\cdot\overline{\overline{x_i\cdots x_j}}\cdot x_{j+1}\cdots x_k
$$
for some $i,j$ with $1\le i\le j\le k$. Using Lemma~\ref{x*=0} we obtain that the variety $\mathbf V$ satisfies the identities
$$
x_1\cdots x_{i-1}\cdot\overline{\overline{x_i\cdots x_j}}\cdot x_{j+1}\cdots x_k\approx 0\approx\mathbf v.
$$
On the other hand, $\mathbf U$ satisfies the identities
$$
x_1\cdots x_{i-1}\cdot\overline{\overline{x_i\cdots x_j}}\cdot x_{j+1}\cdots x_k\approx x_1x_2\cdots x_k=\mathbf{u\approx v}.
$$
Thus, the identity
$$
x_1\cdots x_{i-1}\cdot\overline{\overline{x_i\cdots x_j}}\cdot x_{j+1}\cdots x_k\approx\mathbf v
$$
holds in $\mathbf{V\vee U=V\vee W}$ and therefore, in $\mathbf W$. Then the identities
$$
\mathbf u=x_1x_2\cdots x_k\approx x_1\cdots x_{i-1}\cdot\overline{\overline{x_i\cdots x_j}}\cdot x_{j+1}\cdots x_k\approx\mathbf v
$$
hold in $\mathbf W$.

\smallskip

\emph{Subcase}~2.2: $\mathbf v$ is a semigroup word and $\ell(\mathbf v)>k$. Then equating $x_1,x_2,\dots,x_k$ to $x$, we obtain an identity of the form $x^k\approx x^n$ for some $k<n$. We see that the variety $\mathbf U$ is periodic. Therefore, the variety $\mathbf{V\vee U=V\vee W}$ is periodic too, whence the variety $\mathbf W$ also has this property. Thus, all varieties under consideration are periodic and therefore, can be considered as semigroup varieties. Now we can repeat literally arguments from the proof of sufficiency of Theorem~1.1 in the article~\cite{Gusev-Skokov-Vernikov-18} and conclude that both the varieties $\mathbf U$ and $\mathbf W$ satisfy the identity
$$
x_1x_2\cdots x_k\approx x_1\cdots x_{i-1}(x_i\cdots x_j)^mx_{j+1}\cdots x_k
$$
for some natural $m>1$ and $i,j$ with $1\le i\le j\le k$. Now we can complete the proof by the same arguments as in Subcase~2.1 but with using of the right-hand side of the last identity rather than the word
$$
x_1\cdots x_{i-1}\cdot\overline{\overline{x_i\cdots x_j}}\cdot x_{j+1}\cdots x_k.
$$

\smallskip

\emph{Subcase}~2.3: $\ell(\mathbf v)=k$ or, equivalently, the identity $\mathbf{u\approx v}$ is permutational. Thus, this identity has the form $p_k[\sigma]$ for some $\sigma\in S_k$. By the hypothesis, the variety $\mathbf V$ either satisfies all permutational identities of length $k$ or satisfies none of them. In the former case, the identity $\mathbf{u\approx v}$ holds in $\mathbf V$ and we are done. 

It remains to consider the case when any permutational identity of length $k$ fails in $\mathbf V$. Then $\mathbf V$ coincides with one of the varieties $\mathbf X_{m,n}$ or $\mathbf Y_{m,n}$ with $k<m\le n\le\infty$. Therefore, any non-trivial identity of the form
\begin{equation}
\label{linear word equals w}
x_1x_2\cdots x_k\approx\mathbf w
\end{equation}
fails in $\mathbf V$.

Suppose that $\mathbf W$ contains nilpotent semigroups of nilpotency degree $>k$. We are going to check that in this case $\mathbf W$ does not satisfy any non-trivial and non-permutational identity of the form~\eqref{linear word equals w}. We note that every epigroup variety contains a greatest nil-subvariety. Namely, Lemma~\ref{x*=0} implies that $\mathbf X\wedge\var\{\,\overline x\,\approx 0\}$ is a greatest nil-subvariety of an epigroup variety $\mathbf X$. Let $\mathbf K$ be a greatest nil-subvariety of $\mathbf W$. Then the identity $x_1x_2\cdots x_k\approx 0$ fails in $\mathbf K$. Suppose that $\mathbf W$ satisfies a non-trivial identity of the form~\eqref{linear word equals w}. Then this identity holds in $\mathbf K$. Lemma~\ref{split}(i) implies that our identity is permutational. 

The identity $p_k[\sigma]$ holds in $\mathbf U$. Hence it holds in $\mathbf{V\wedge U=V\wedge W}$. Let the sequence of words
$$
x_1x_2\cdots x_k=\mathbf w_0,\mathbf w_1,\dots,\mathbf w_s= x_{1\sigma}x_{2\sigma}\cdots x_{k\sigma}
$$
be a deduction of shortest length of the identity $p_k[\sigma]$ from the identities of the varieties $\mathbf V$ and $\mathbf W$. The identity $\mathbf w_0\approx\mathbf w_1$ fails in $\mathbf V$. Hence it holds in $\mathbf W$. Therefore, this identity is permutational, whence $\mathbf w_1$ is a linear word of length $k$. If $s>1$ then the identity $\mathbf w_1\approx\mathbf w_2$ holds in $\mathbf V$ but this is impossible. Therefore, $s=1$. This means that $\mathbf W$ satisfies the identity $ x_1x_2\cdots x_k\approx\mathbf v$.

It remains to consider the case when all nilpotent semigroups in $\mathbf W$ have nilpotency degree $\le k$. Recall that $\mathbf V$ is a nil-variety. Hence the variety $\mathbf{V\wedge W=V\wedge U}$ satisfies the identity $x_1x_2\cdots x_k\approx 0$ and therefore, the identity $x_1x_2\cdots x_k\approx x_1x_2\cdots x_ky$ for any letter $y$. Let
$$
x_1x_2\cdots x_k=\mathbf p_0,\mathbf p_1,\dots,\mathbf p_t= x_1x_2\cdots x_ky
$$
be a deduction of the last identity from the identities of the varieties $\mathbf V$ and $\mathbf U$. There is an index $i$ such that the identity $\mathbf p_{i-1}\approx\mathbf p_i$ is non-permutational. Let $i$ be the least index with this a property. Clearly, $i>0$ and $\mathbf p_{i-1}$ is a linear word with $\con(\mathbf p_{i-1})=\{x_1,x_2,\dots,x_k\}$. The identity $\mathbf p_{i-1}\approx\mathbf p_i$ holds in either $\mathbf V$ or $\mathbf U$. Since any non-permutational identity of the form~\eqref{linear word equals w} fails in $\mathbf V$, some identity of such a form holds in $\mathbf U$. Note that $\con(\mathbf w)=\{x_1,x_2,\dots,x_k\}$ because $\mathbf{U\supseteq SL}$. In particular, we have that the word $\mathbf w$ is non-linear. The identity~\eqref{linear word equals w} is equivalent to the identity $x_{1\sigma}x_{2\sigma}\cdots x_{k\sigma}\approx\sigma[\mathbf w]$ for any $\sigma\in S_k$. By Subcases~2.1 and~2.2, both the last identity and~\eqref{linear word equals w} hold in $\mathbf W$. The variety $\mathbf U$ satisfies the identities
$$
\mathbf w\approx x_1x_2\cdots x_k\approx x_{1\sigma}x_{2\sigma}\cdots x_{k\sigma}\approx\sigma[\mathbf w].
$$
As we noted above, the word $\mathbf w$ is non-linear. Therefore, the identities $\mathbf w\approx 0\approx\sigma[\mathbf w]$ hold in $\mathbf V$ by~\eqref{xxy=xyx=yxx=0} whenever $\mathbf w$ is a semigroup word or by Lemma~\ref{x*=0} otherwise. Hence the identity $\mathbf w\approx\sigma[\mathbf w]$ holds in the variety $\mathbf{V\vee U=V\vee W}$ and therefore, in $\mathbf W$. Now we see that $\mathbf W$ satisfies the identities
$$
x_1x_2\cdots x_k\approx\mathbf w\approx\sigma[\mathbf w]\approx x_{1\sigma}x_{2\sigma}\cdots x_{k\sigma}.
$$
Thus, the identity $p_k[\sigma]$ holds in $\mathbf W$.

Theorem~\ref{main epigroup} is proved.\qed

\smallskip

\emph{Proof of Theorem}~\ref{main semigroup}. \emph{Sufficiency}. The scheme of our considerations here is the same as in the proof of sufficiency of Theorem~\ref{main epigroup}. Suppose that a semigroup variety $\mathbf V$ has the form indicated in the formulation of Theorem~\ref{main semigroup}. We need to verify that $\mathbf V$ is a cancellable element of $\mathbb{SEM}$. The case when $\mathbf{V=SEM}$ is evident. Let now $\mathbf{V\ne SEM}$ and $\mathbf U$ and $\mathbf W$ be semigroup varieties such that $\mathbf{V\vee U=V\vee W}$ and $\mathbf{V\wedge U=V\wedge W}$. We need to check that $\mathbf{U=W}$. Corollary~\ref{join with SL} allows us to suppose that $\mathbf V$ is one of the varieties $\mathbf T$, $\mathbf X_{m,n}$ or $\mathbf Y_{m,n}$ with $2\le m\le n\le\infty$. In particular, the variety $\mathbf V$ is periodic. The case when $\mathbf{V=T}$ is evident. By symmetry, it suffices to show that a non-trivial identity $\mathbf{u\approx v}$ holds in $\mathbf W$ whenever it holds in $\mathbf U$. As in the proof of sufficiency of Theorem~\ref{main epigroup}, Corollary~\ref{join with SL} and Lemma~\ref{identities in SL} allow us to assume that $\con(\mathbf u)=\con(\mathbf v)$.

If $|\con(\mathbf u)|=1$ then it suffices to refer to arguments given in Subcase~1.1 from the proof of sufficiency of Theorem~\ref{main epigroup}. Suppose now that $|\con(\mathbf u)|=k>1$. As in Case 2 from the proof of sufficiency of Theorem~\ref{main epigroup}, we can assume that $\mathbf u=x_1x_2\cdots x_k$. If $\ell(\mathbf v)>k$ then we can complete the proof by the same arguments as in Subcase~2.2 of the proof of sufficiency of Theorem~\ref{main epigroup}.

It remains to consider the case when $\ell(\mathbf v)=k$. A semigroup variety is called \emph{overcommutative} if it contains the variety of all commutative semigroups. It is well known that any semigroup variety is either periodic or overcommutative. As in the proof of sufficiency of Theorem~\ref{main epigroup}, the varieties $\mathbf U$ and $\mathbf W$ are either both periodic or both non-periodic, i.e., overcommutative. Suppose at first that $\mathbf U$ and $\mathbf W$ are periodic. Hence the variety $\mathbf W$ contains the greatest nil-subvariety. In this case, we can complete the proof by the same arguments as in the third and the fourth paragraphs of Subcase~2.3 of the proof of sufficiency of Theorem~\ref{main epigroup}. 

Finally, suppose that $\mathbf U$ and $\mathbf W$ are overcommutative. It is well known and easy to check that if an overcommutative variety satisfies some identity then each letter occurs in both sides of this identity the same number of times. Therefore, if $\mathbf W$ satisfies an identity of the form~\eqref{linear word equals w} then this identity is permutational. Further, by the hypothesis, the variety $\mathbf V$ either satisfies all permutational identities of length $k$ or satisfies none of them. As in the proof of sufficiency of Theorem~\ref{main epigroup}, in the former case we are done. It remains to consider the case when any permutational identity of length $k$ fails in $\mathbf V$. Then $\mathbf V$ coincides with one of the varieties $\mathbf X_{m,n}$ or $\mathbf Y_{m,n}$ with $k<m\le n\le\infty$. Therefore, any non-trivial identity of the form~\eqref{linear word equals w} fails in $\mathbf V$. Now we can complete the proof by repeating literally the fourth paragraph of Subcase~2.3 of the proof of sufficiency of Theorem~\ref{main epigroup}.

Theorem~\ref{main semigroup} is proved.\qed

\subsubsection*{Acknowledgements}
In September--October 2013, the third author visited the University of Novi Sad by the invitation of Professor Igor Dolinka. During this visit, Professors \v{S}e\v{s}elja and Tepav\v{c}evi\'{c} asked him a question: what is known about cancellable elements of the lattice of semigroup varieties? At that time, nothing was known about this. After that, the third author and his colleagues were studying this question and its epigroup modification. The first results of this study were presented in~\cite{Gusev-Skokov-Vernikov-18,Skokov-18,Skokov-Vernikov-19}, while here we provide the exhaustive answer to the question. We thank Professor Igor Dolinka for the excellent organization of the visit and Professors \v{S}e\v{s}elja and Tepav\v{c}evi\'{c} for their interest in our work and such a fruitful discussion.


\small

\begin{thebibliography}{99}
\bibitem{Aizenstat-Boguta-79}
A\v{\i}zen\v{s}tat, A.Ya., Boguta, B.K. (1979). On the lattice of semigroup varieties. In E.~S.~Lyapin, ed., \emph{Semigroup Varieties and Endomorphism Semigroups}. Leningrad: Leningrad State Pedagogic Institute, pp.3--46 [Russian].
\bibitem{Burris-Nelson-71}
Burris, S., Nelson, E. (1971). Embedding the dual of $\Pi_\infty$ in the lattice of equational classes of semigroups. \emph{Algebra Universalis}. 1(2):248--254.
\bibitem{Evans-71}
Evans, T. (1971). The lattice of semigroup varieties. \emph{Semigroup Forum}. 2(1):1--43. DOI: https://doi.org/10.1007/BF02572269.
\bibitem{Gratzer-11}
Gr\"atzer, G. (2011). \emph{Lattice Theory: Foundation}. Birkh\"auser, Springer Basel AG. DOI: 10.1007/978-3-0348-0018-1.
\bibitem{Gratzer-Schmidt-61}
Gr\"atzer, G., Schmidt, E.T. (1961). Standard ideals in lattices. \emph{Acta Math. Acad. Sci. Hungar.} 12(1--2):17--86. DOI: 10.1007/BF02066675.
\bibitem{Gusev-Skokov-Vernikov-18}
Gusev, S.V., Skokov, D.V., Vernikov, B.M. (2018). Cancellable elements of the lattice of semigroup varieties. \emph{Algebra and Discr. Math.} 26(1):34--46.
\bibitem{Gusev-Vernikov-16}
Gusev, S.V., Vernikov, B.M. (2016). Endomorphisms of the lattice of epigroup varieties. \emph{Semigroup Forum}. 93(3):554--574. DOI: 10.1007/s00233-016-9825-6.
\bibitem{Jezek-76}
Je\v{z}ek, J. (1976). Intervals in lattices of varieties. \emph{Algebra Universalis}. 6(1):147--158. DOI: https://doi.org/10.1007/BF02485826.
\bibitem{Jezek-81}
Je\v{z}ek, J. (1981). The lattice of equational theories. Part~I: modular elements. \emph{Czechosl. Math. J.} 31(1):127--152.
\bibitem{Jezek-McKenzie-93}
Je\v{z}ek, J., McKenzie, R.N. (1993). Definability in the lattice of equational theories of semigroups. \emph{Semigroup Forum}. 46(1):199--245. DOI: 10.1007/BF02573566.
\bibitem{Kalicki-Scott-55}
Kalicki, J., Scott, D. (1955). Equationally completeness in abstract algebras. \emph{Proc. Konikl. Nederl. Akad. Wetensch. Ser. A}. 58(17):650--659.
\bibitem{Mitsch-83}
Mitsch, H. (1983). Semigroups and their lattice of congruences. \emph{Semigroup Forum}. 26(1):1--63. DOI: 10.1007/BF02572819.
\bibitem{Petrich-Reilly-99}
Petrich, M., Reilly, N.R. (1999). \emph{Completely Regular Semigroups}. New York: John Wiley \& Sons, Inc.
\bibitem{Sapir-Sukhanov-81}
Sapir, M.V., Sukhanov, E.V. (1981). Varieties of periodic semigroups. \emph{Izv. VUZ. Matematika}. No.4:48--55 [Russian; Engl. translation: \emph{Soviet Math. Iz. VUZ}. 25(4):53--63]. 
\bibitem{Shaprynskii-12}
Shaprynski\v{\i}, V.Yu. (2012). Modular and lower-modular elements of lattices of semigroup varieties. \emph{Semigroup Forum}. 85(1):97--110. DOI: 10.1007/s00233-012-9385-3.
\bibitem{Shaprynskii-Skokov-Vernikov-16}
Shaprynski\v{\i}, V.Yu., Skokov, D.V., Vernikov, B.M. (2016). Special elements of the lattice of epigroup varieties. \emph{Algebra Universalis}. 76(1):1--30. DOI: 10.1007/s00012-016-0380-5.
\bibitem{Shevrin-94}
Shevrin, L.N. (1994). On the theory of epigroups.~I,~II. \emph{Matem. Sborn.} 185(8):129--160, 185(9):153--176 [Russian; Engl. translation (1995): \emph{Russ. Acad. Sci. Sb. Math.} 82(2):485--512, 83(1):133--154]. DOI: 10.1070/SM1995v082n02ABEH003577, 10.1070/SM1995v083n01ABEH003584.
\bibitem{Shevrin-05}
Shevrin, L.N. (2005). Epigroups. In: Kudryavtsev, V.B., Rosenberg, I.G., eds. \emph{Structural Theory of Automata, Semigroups, and Universal Algebra}. Dordrecht: Springer, pp.331--380. DOI: 10.1007/1-4020-3817-8\underline{\phantom{0}}12.
\bibitem{Shevrin-Vernikov-Volkov-09}
Shevrin, L.N., Vernikov, B.M., Volkov, M.V. (2009). Lattices of semigroup varieties, \emph{Izv. VUZ. Matematika}. No.3:3--36 [Russian; Engl. translation: \emph{Russ. Math. Izv. VUZ}. 53(3):1--28]. DOI: 10.3103/S1066369X09030013.
\bibitem{Skokov-15}
Skokov, D.V. (2015). Distributive elements of the lattice of epigroup varieties. \emph{Siberian Electronic Math. Reports}. 12:723--731 [Russian]. DOI: 10.17377/semi.2015.12.058.
\bibitem{Skokov-16}
Skokov, D.V. (2016). Special elements of certain types in the lattice of epigroup varieties. \emph{Proc. Institute of Math. and Mechan. of the Ural Branch of the Russ. Acad. Sci.} 22(3):244--250 [Russian]. DOI: 10.21538/0134-4889-2016-22-3-244-250.
\bibitem{Skokov-18}
Skokov, D.V. (2018). Cancellable elements of the lattice of epigroup varieties. \emph{Izv. VUZ. Matem.} No.9:59--67 [Russian; Engl. translation: \emph{Russ. Math. Iz. VUZ}. 62(9):52--59]. DOI: 10.3103/S1066369X18090062.
\bibitem{Skokov-Vernikov-19}
Skokov, D.V., Vernikov, B.M. (2019). On modular and cancellable elements of the lattice of semigroup varieties. \emph{Siberian Electronic Math. Reports}. 16:175--186. DOI: 10.33048/semi.2019.16.010
\bibitem{Vernikov-07}
Vernikov, B.M. (2007). On modular elements of the lattice of semigroup varieties. \emph{Comment. Math. Univ. Carol.} 48(4):595--606.
\bibitem{Vernikov-15}
Vernikov, B.M. (2015). Special elements in lattices of semigroup varieties. \emph{Acta Sci. Math.} (\emph{Szeged}). 81(1--2):79--109. DOI: 10.14232/actasm-013-072-0.
\bibitem{Vernikov-Shaprynskii-10}
Vernikov, B.M., Shaprynski\v{\i}, V.Yu. (2010). Distributive elements of the lattice of semigroup varieties. \emph{Algebra i Logika}. 49(3):303--330 [Russian; Engl. translation: \emph{Algebra and Logic}. 49(3):201--220]. DOI: 10.1007/s10469-010-9090-9.
\bibitem{Volkov-05}
Volkov, M.V. (2005). Modular elements of the lattice of semigroup varieties. \emph{Contrib. General Algebra}. 16:275--288. DOI: 10.1007/s10469-010-9090-9.
\end{thebibliography}
\end{document}